 \documentclass[leqno]{amsart}
\usepackage[T1]{fontenc}
\usepackage{amsfonts,amssymb,amsmath,amsgen,amsthm}
\usepackage{hyperref,color}
\usepackage{pdfsync}
\usepackage{mathtools}
\usepackage{orcidlink}
\makeatletter
\theoremstyle{plain}
\newtheorem{thm}{\protect\theoremname}
\theoremstyle{definition}
\newtheorem{defn}[thm]{\protect\definitionname}
\theoremstyle{remark}
\newtheorem{rem}[thm]{\protect\remarkname}
\theoremstyle{plain}
\newtheorem{lem}[thm]{\protect\lemmaname}
\theoremstyle{plain}
\newtheorem{prop}[thm]{\protect\propositionname}

\makeatother

\usepackage{babel}
\providecommand{\definitionname}{Definition}
\providecommand{\lemmaname}{Lemma}
\providecommand{\propositionname}{Proposition}
\providecommand{\remarkname}{Remark}
\providecommand{\theoremname}{Theorem}

\def\R{{\mathbf R}}% real numbers
\def\T{{\mathbb T}}% torus
\def\N{{\mathbf N}}% nonnegative integers
\def\Z{{\mathbf Z}}% integers
\def\d{{\partial}}
\def\eps{\varepsilon}
\DeclareMathOperator{\RE}{Re}
\DeclareMathOperator{\IM}{Im}

\numberwithin{equation}{section}

\date\today

\title[QHD system with damping and relaxation-time limit]{The time-relaxation limit for weak solutions to the quantum hydrodynamics system}

\author[P. Antonelli]{Paolo Antonelli \,\orcidlink{0000-0001-8211-8296}}
\address{Gran Sasso Science Institute, viale Francesco Crispi, 7, 67100 L'Aquila, Italy}
\email{paolo.antonelli@gssi.it}

\author[P. Marcati]{Pierangelo Marcati*\,\orcidlink{0000-0001-6528-6562}}
\address{Gran Sasso Science Institute, viale Francesco Crispi, 7, 67100 L'Aquila, Italy}
\email{pierangelo.marcati@gssi.it}

\author[H. Zheng]{Hao Zheng\,\orcidlink{0000-0002-7098-7731}}
\address{Chinese Academy of Science, Zhongguancun Est. Rd., Haidian District, Beijing, China, 100190}
\email{zhenghao@amss.ac.cn}


\begin{document}
\begin{abstract}
This paper analyzes weak solutions of the quantum hydrodynamics (QHD) system with 
%linear damping 
a collisional term posed on the one-dimensional torus. The main goal of our analysis is to rigorously prove the time-relaxation limit towards solutions to the quantum drift-diffusion (QDD) equation.
\newline
The existence of global in time, finite energy weak solutions can be proved by straightforwardly exploiting the polar factorization and wave function lifting tools previously developed by the authors. However, the sole energy bounds are not sufficient to show compactness and then pass to the limit.
\newline
For this reason, we consider a class of more regular weak solutions (termed GCP solutions), determined by the finiteness of a functional involving the chemical potential associated with the system. For solutions in this class and bounded away from vacuum, we prove the time-relaxation limit and provide an explicit convergence rate.
\newline
Our analysis exploits compactness tools and does not require the existence (and smoothness) of solutions to the limiting equations or the well-preparedness of the initial data.
\newline
As a by-product of our analysis, we also establish the existence of global in time $H^2$ solutions to a nonlinear Schr\"odinger-Langevin equation and construct solutions to the QDD equation as strong limits of GCP solutions to the QHD system.
\end{abstract}
\maketitle
\section{Introduction}\label{sect:intro}

The main object of this paper is to study the relaxation limit of weak solutions for the collisional one-dimensional quantum hydrodynamic (QHD) system,
\begin{equation}\label{eq:QHD}
\begin{cases}
\d_{t}\rho+\d_xJ=0\\
\d_{t}J+\d_x\left(\frac{J^2}{\rho}\right)+\d_x p(\rho)+\rho\d_x V=\frac{1}{2}\rho\d_x\left(\frac{\d_x^2\sqrt\rho}{\sqrt\rho}\right)-\frac{1}{\tau}J\\
-\d_x^2V=\rho-\mathcal{C}(x),\quad (\rho,J)(0,x)=(\rho_0,J_0)(x),
\end{cases}
\end{equation}
for $(t,x)\in[0,T)\times\T$, $T>0$. $\T$ denotes the one-dimensional torus that, without loss of generality, we assume $\T=\R/\Z$. The unknowns $\rho$ and $J$ in \eqref{eq:QHD} represent the mass (charge) and the momentum (current) densities of the quantum fluid, respectively, and $p(\rho)$ is the isentropic pressure term. The parameter $\tau>0$ denotes the scaled momentum relaxation time.

In order to study the relaxation limit, we also provide a global existence result for weak solutions to \eqref{eq:QHD} and a global wellposedness result for the associated Schr\"odinger-Langevin equation (see \eqref{eq:NLS} below) for non-vacuum initial states. 
Furthermore, as a by-product of our analysis we also show a global existence result for solutions of the limiting equation (Quantum Drift-Diffusion equation, see \eqref{eq:qdde} below).
For this reason, our analysis does not need to assume any property on the solutions of the limiting equations, as usually done with relative energy/entropy techniques, for instance.

Among several applications, model \eqref{eq:QHD} can be used also to model carrier transport in semiconductor devices \cite{J}, where the function $V$ denotes a self-consistent electric potential, satisfying the Poisson equation, and $\mathcal{C}(x)$ denotes the density of background positively charged ions.
It is also possible to consider, with similar analysis, different Poisson couplings, e.g. repulsive potentials, but we will not develop this aspect in detail.
When the doping profile is relevant at nanoscales, quantum effects must be taken into account \cite{AI, G}.
\newline
The usual stress tensor for quantum fluids takes the following form \cite{BP}
\begin{equation*}
\Pi = -p(\rho) + \frac14 (\rho\d_x^2\log\rho),
\end{equation*}
whereas its quantum part  satisfies the following identities
\begin{equation}\label{eq:bohm}
\frac12\rho\d_x\left(\frac{\d_{x}^2\sqrt{\rho}}{\sqrt{\rho}}\right)=\frac14\d_x(\rho\d_x^2\log\rho)=\frac14\d_{x}^3\rho-\d_x[(\d_x\sqrt{\rho})^2].
\end{equation}
The left hand side is commonly interpreted as a quantum pressure expressed in terms of a quantum enthalpy, while the right hand side provides the decoupling into a linear dispersive tensor and a quadratic term taking into account the Fisher information \cite{HC}.
%\textcolor{red}{general physics reference}.
\newline
Quantum fluid models of this type have applications in various physical phenomena, in particular where the thermal de Broglie wavelength is comparable with the interatomic distance.
For instance, this is the case in superfluidity of Helium II \cite{K}, Bose-Einstein condensates \cite{PS} and quantum plasmas \cite{H}. 
\newline 
Besides its relevance in semiconductor device modeling, the collisional model \eqref{eq:QHD} with $\tau>0$ can also be seen as a toy model to study the interaction of a quantum fluid with a classical one, as considered in the two-fluid model proposed by Landau and Tisza, see the textbooks \cite[Section 17]{K} and \cite[Chapter XVI]{LL}. 
The derivation of the collisional term, in the case of semiconductor modeling, can be obtained either by using a phenomenological approach, taking into account the thermal interactions between charged particles (see  \cite{AI,AT,G} or Section 7 in \cite{AM_b}), or by using kinetic-type models as an approximation of intra-band collisions, see \cite{BW}.

The main goal of this paper is the study of the relaxation-time limit of system \eqref{eq:QHD}. More precisely, by introducing the scaling \cite{MN} (see also \cite{DM, MMS} for more general discussion),
\begin{equation}\label{eq:rs_intro}
t'=\tau t,\quad (\rho_\tau,J_\tau)(t',x)=\left(\rho,\frac{1}{\tau}J\right)\left(\frac{t'}{\tau},x\right),
\end{equation}
the system \eqref{eq:QHD} can be reformulated as
\begin{equation}\label{eq:QHD_rs_intro}
\left\{\begin{aligned}
&\d_{t'}\rho_\tau+\d_x J_\tau=0\\
&\tau^2\d_{t'} J_\tau+\tau^2\d_x\left(\frac{J_\tau^2}{\rho_\tau}\right)+\d_x p(\rho_\tau)+\rho_\tau\d_xV_\tau=\frac{1}{2}\rho_\tau\d_x\left(\frac{\d_x^2\sqrt\rho_\tau}{\sqrt\rho_\tau}\right)-J_\tau\\
&-\d_x^2 V_\tau=\rho_\tau-\mathcal{C}(x),\quad (\rho_\tau,J_\tau)(0,x)=(\rho_0,J_{\tau,0})(x),
\end{aligned}\right.
\end{equation}
As $\tau\to 0$, system \eqref{eq:QHD_rs_intro} formally converges to the following Quantum Drift-diffusion (QDD) equation, 
\begin{equation}\label{eq:qdde}
\begin{cases}
\d_{t'}\bar\rho+\d_x\left[\frac{1}{2}\bar\rho\d_x\left(\frac{\d_x^2\sqrt{\bar\rho}}{\sqrt{\bar\rho}}\right)-\d_x p(\bar\rho)-\bar\rho\d_x\bar{V}\right]=0\\
-\d_x^2\bar V=\bar\rho-\mathcal{C}(x),\quad \rho_\tau(0,x)=\rho_0(x),
\end{cases}
\end{equation}
see for instance \cite{J, R}.
%In certain regimes, the QDD model is used as a reduced model of the collisional QHD system. For instance in the context of semiconductor devices, quantum transport of electrons at certain intermediate time scales can be effectively described by the quantum drift-diffusion (QDD) equation \cite{J,R}. Such convergence is often referred as the \emph{relaxation-time limit} of the quantum hydrodynamic system \eqref{eq:QHD}. 
\newline
The main goal of our paper is the rigorous justification of this limit in a suitable functional framework that will be explained below.

\subsection{Physical framework and weak solutions}

A natural framework to study system \eqref{eq:QHD} is the space of states with finite mass and energy, characterised by the total mass
\begin{equation}\label{eq:mass}
M(t)=\int_\T \rho(t)dx\equiv M_0,
\end{equation}
and total energy functional
\begin{equation}\label{eq:en}
E(t)=\int_\T \frac12(\d_x\sqrt\rho)^2+\frac{J^2}{2\rho}+f(\rho)+\frac12(\d_xV)^2dx.
\end{equation}
The internal energy $f(\rho)$, appearing in \eqref{eq:en}, is related to the pressure $p(\rho)$ via the relation
\begin{equation}\label{eq:pres}
p(\rho)=f'(\rho)\rho-f(\rho),
\end{equation}
or equivalently
\[
f(\rho)=\rho\int^\rho \frac{p(s)}{s^2}ds.
\]
In this paper we will consider constitutive relations for the pressure $p=p(\rho)$ such that $p\in C^2((0,\infty))\cap C^1([0, \infty))$ and $f=f(\rho)\geq0$. 
Moreover, in Theorem \ref{thm:disp} below we will restrict to classical $\gamma-$law equations of state for the pressure. 
In the Poisson equation in \eqref{eq:QHD} satisfied by the  electric potential $V$, we assume $\mathcal{C}(x)\in L^1(\T)$ such that $\int_\T\mathcal{C}(x)dx=M_0$, which guarantees the solvability of the Poisson equation. 
By combining the authors' previous works \cite{AM1, AM2} and \cite{AMZ1}, it is possible to prove an existence theorem for global in time finite energy weak solutions, that is stated below here for completeness. 

\begin{prop}[Global Existence of finite energy weak solutions]\label{thm:glob} 
Let $(\rho_0, J_0)$ be a finite energy initial datum such that,
\begin{equation}\label{eq:C1_intro}
\|\sqrt{\rho_0}\|_{H^1}^2+\|J_0/\sqrt{\rho_0}\|_{L^2}^2\leq C
\end{equation}
and let us further assume that $J_0/\sqrt{\rho_0}=0$ a.e. on the set $\{\rho_0=0\}$. Then, there exists a global in time finite energy weak solution to the Cauchy problem that satisfies
\begin{equation}\label{eq:B1_1}
E(t)+\frac1\tau\int_0^t\int_{\T}\frac{J^2}{\rho}\,dxds\leq C E(0).
% \|\sqrt{\rho}\|_{L^\infty(0, T;H^1(\T))}+\|\Lambda\|_{L^\infty(0, T;L^2(\T))}\leq C.
\end{equation}
\end{prop}
This Proposition provides the existence of arbitrarily large finite energy weak solutions, without further assumptions on their structure. However, the (scaled) uniform bounds in \eqref{eq:B1_1} are not sufficient to pass to the limit in \eqref{eq:QHD_rs_intro}. Indeed, while the inertial term can be controlled by the scaled energy dissipation, the weak limit of the the quadratic term $(\d_x\sqrt{\rho_\tau})^2$ (see \eqref{eq:bohm} above) may be affected by strong oscillations and concentration phenomena, which cannot be controlled by the sole energy bounds.

Therefore, in order to avoid these pathological behaviors, we exploit the compactness framework developed in the authors' previous work \cite{AMZ1} and introduce the notion of weak solutions with bounded generalised chemical potential, shortened as \emph{GCP solutions}, that are thoroughly described in Definition \ref{def:GCPsln} below. The chemical potential is formally defined as
\begin{equation}\label{eq:chem}
\mu=-\frac{\d_x^2\sqrt\rho}{2\sqrt\rho}+\frac12\frac{J^2}{\rho^2}+f'(\rho)+V
\end{equation}
and it can be formally interpreted as the first variation of the total energy functional with respect to the mass density,
\[
\mu=\frac{\delta E}{\delta \rho}.
\]
Indeed, by noticing that $J=\frac{\delta E}{\delta v}$, then equations \eqref{eq:QHD} can be formally written as the following system (Hamilton system with a dissipative term formally given by $v=\frac{J}{\rho}$) associated to the energy functional
\begin{equation}\label{eq:Ham_form}
\left\{\begin{aligned}
&\d_t\rho=-\d_x \frac{\delta E}{\delta v}\\
&\d_t v=-\d_x \frac{\delta E}{\delta \rho}-\frac{1}{\tau}v
=-\d_x\mu-\frac1\tau v.
\end{aligned}\right.
\end{equation}
By using the objects introduced above, we can define the following functional
\begin{equation}\label{eq:higher}
I(t)=\int_\T \frac{\rho}{2}(\mu^2+\sigma^2)dx,
\end{equation}
where 
\begin{equation}\label{eq:sigma}
\sigma=\d_t\log\sqrt\rho=-\frac{\d_xJ}{2\rho}.
\end{equation} 
$I(t)$ somehow may be interpreted as a higher order functional for \eqref{eq:QHD}. In \cite{AMZ1}, it is shown that $I(t)$ remains uniformly bounded over compact time intervals, for the Hamiltonian evolution. Moreover, $I(t)$ also identifies a compactness class for weak solutions to QHD.
Following the terminology used in \cite{AMZ1}, we denote by GCP solutions, the family of weak solutions with finite energy, see \eqref{eq:en}, and functional $I(t)$, defined in \eqref{eq:higher}, see also Definition \ref{def:GCPsln}.
Unlike \cite{AMZ1} where vacuum regions are allowed, 
our analysis requires to consider solutions whose mass density remains uniformly bounded away from vacuum.
Indeed, by using \eqref{eq:QHD}, the time derivative of $I(t)$ is formally given by
%\textcolor{blue}{This may be seen, for instance, by computing the time derivative of $I(t)$,}
%newline
%By using \eqref{eq:QHD} and the definitions \eqref{eq:chem}, \eqref{eq:sigma}, it is then possible to formally compute the time derivative of $I(t)$, that reads
%\newline
\begin{equation}\label{eq:dtI_intro}
\frac{d}{dt} I(t)+\frac{1}{\tau}\int_T\rho \sigma^2 dx=\int_\T \mu\d_tp(\rho)dx+\int_\T\rho\mu \d_tVdx-\frac{1}{\tau}\int_\T\rho v^2\mu dx.
\end{equation}
In particular, the last term in the right hand side of \eqref{eq:dtI_intro} appears due to the collision term. The chemical potential may experience strong density fluctuations close to vacuum,  and possibly generate loss of integrability properties. Therefore in this paper we restrict to strictly positive densities. The rigorous computation of $\frac{d}{dt} I(t)$ will be given in Proposition \ref{prop:3.2}.

\subsection{Main result: global well-posedness}

The first result of this paper regards the global well-posedness of GCP weak solutions with strictly positive density. For a more clear definition of the GCP solutions, we refer to Definition \ref{def:GCPsln}.

\begin{thm}[Global well-posedness of GCP weak solutions]\label{thm:glob2}
%Let $\mathcal C$ be given and satisfying $\int_{\T}\mathcal C(x)dx=M_0$.
Let us consider a finite energy initial datum $(\rho_0, J_0)$ satisfying the following conditions.
\begin{itemize}
\item Let us denote by $M_0=\int_\T\rho_0dx<\infty$ the initial total mass, and by $v_0=J_0/\rho_0$ the initial velocity. We assume they satisfy 
\begin{equation}\label{eq:cond_1}
\int_\T \mathcal{C}(x)dx=M_0,\quad \int_\T v_0 dx=0.
\end{equation}

\item There exists a $0<\delta\leq M_0$ such that the initial energy satisfies
\begin{equation}\label{eq:ini_en_intro}
E_0=\int_\T \frac12(\d_x\sqrt\rho_0)^2+\frac12\rho_0 v_0^2+f(\rho_0)+\frac12(\d_x V_0)^2dx\leq \frac12\left(M_0^\frac12-\delta M_0^{-\frac12}\right)^2,
\end{equation}
where $V_0$ is given by $-\d_x^2 V_0=\rho_0-\mathcal{C}(x)$.
\item The initial data satisfy the bounds
\begin{equation}\label{eq:C2_intro}
\|\sqrt{\rho_0}v_0^2-\d_x^2\sqrt{\rho_0}\|_{L^2_x}
+\|\frac{\d_xJ_0}{\sqrt{\rho_0}}\|_{L^2_x}\leq I_0<\infty.
\end{equation}
\end{itemize}
Then, there exists a unique global in time GCP solution $(\rho,J)$ to \eqref{eq:QHD}, such that for any $0\leq t<\infty$
\begin{itemize}
\item the density remains strictly positive with lower bound $\inf_{t,x}\rho\geq \delta$;
\item the total energy $E(t)$ satisfies the energy balance law, namely for every $t>0$,  
\begin{equation}\label{eq:en_disp_intro}
E(t)+\frac{1}{\tau}\int_0^t\int_\T \rho v^2 dxds=E_0;
\end{equation}
\item the higher order functional $I(t)$ satisfies the bounds
\begin{equation}\label{eq:bd_I}
I(t)+\frac{1}{\tau}\int_0^t\int_\T [(\d_t\sqrt\rho)^2+\rho v^4 ]dxds\leq C(M_0,E_0,I_0,\delta,\tau t)
\end{equation}
where $C(\dots,\tau t)$ grows at most at the rate $e^{\tau t}$.
\end{itemize}
\end{thm}

\begin{rem}
The condition \eqref{eq:ini_en_intro} is introduced to obtain the lower bound $\delta$ of density $\rho_0$ and $\rho$ by using  \eqref{eq:en_disp_intro} and Poincar\'e inequality.
The conditions \eqref{eq:cond_1} in Theorem \ref{thm:glob2} are preserved along the evolution \eqref{eq:QHD}, so they guarantee the solvability of the Poisson equation in $V$ and the periodicity of the phase function $S$ respectively, see also Definition \ref{def:phase} below. Last, the condition \eqref{eq:C2_intro} is equivalent to the boundedness of the functional $I(t)$ at initial time $t=0$, and conversely the bound \eqref{eq:bd_I} implies \eqref{eq:C2_intro} holds upto any finite time $t>0$.
%In the conditions of Theorem \ref{thm:glob2}, we assume $\int_\T v_0 dx=0$, which is preserved by the flow of \eqref{eq:QHD} and guarantees the periodicity of the phase function $S$ (see Definition \ref{def:phase}). 

\end{rem}

Our strategy to prove Theorem \ref{thm:glob2} is motivated by noticing there is a formal equivalence between system \eqref{eq:QHD} and the following Schr\"odinger-Langevin type equation 
\begin{equation}\label{eq:NLS}
i\d_t\psi+\frac12\d_x^2\psi=f'(|\psi|^2)\psi+\frac{1}{\tau}S\psi+V\psi,
\end{equation}
see also \cite{JMR}.
Here the potential $S$ is given through the Madelung transformations \cite{Mad}, $\psi=|\psi|e^{iS}$ as the phase function associated to $\psi$, and it can be formally defined as $S=\frac{1}{i}\log\left(\frac{\psi}{|\psi|}\right)$.
Equation \eqref{eq:NLS} was introduced, in the case $f=0$, by Kostin \cite{Kos} (see also \cite{Y} and \cite{Na}) to derive a Schr\"odinger equation which was taking into account the interaction with Brownian particles. 
%It is also given by Nassar \cite{Na} as a consequence of stochastic interpretation of quantum mechanics. 
%The potential $S$ is interpreted as the phase function associated to $\psi$ through the Madelung transformations \cite{Mad}, $\psi=|\psi|e^{iS}$, which can be formally defined $S=\frac{1}{i}\log\left(\frac{\psi}{|\psi|}\right)$. 
The usual difficulties entailed by the definition of the complex logarithm prevents us to study equation \eqref{eq:NLS} with standard methods.
\newline
%However, the definition of the logarithm function is not clarified, unless the range of $\psi$ is restricted in a simply connected set away from $0$ in the complex plane. 
%\textcolor{red}{Since equation \eqref{eq:NLS} and the function $S$ play an impotent role in our analysis, in Section \ref{sect:def} we will give an explicit definition in hydrodynamic formulation (see Definition \ref{def:phase}). }
The hydrodynamic system \eqref{eq:QHD} can be formally recast from \eqref{eq:NLS} by defining the associated hydrodynamic variables  $(\rho, J)=(|\psi|^2, \IM(\bar\psi\d_x\psi))$ and by computing the related balance laws, see for instance \cite{AMZ1}.
%In the first part of this paper, we will rigorously establish the equivalence between the QHD system \eqref{eq:QHD} and the Schr\"odinger-Langevin equation \eqref{eq:NLS} \textcolor{blue}{under suitable assumptions?}, and prove the global well-posedness of the Cauchy problem of these equations.
%\newline

To prove Theorem \ref{thm:glob2}, we first show the the well-posedness of the Cauchy problem for Schr\"odinger-Langevin equation \eqref{eq:NLS}. We rigorously establish the equivalence between QHD system \eqref{eq:QHD} and equation \eqref{eq:NLS} for solutions with strictly positive density, by the methods of wave function lifting \cite{AMZ1} (see also \cite{HaoMinE}) and polar factorization \cite{AM1} developed in the authors' previous works. As an independent result, we also state the well-posedness of equation \eqref{eq:NLS} as the following theorem.

\begin{thm}\label{thm:NLS}
Let us assume that the initial datum $\psi_0\in H^2_x(\T)$ for \eqref{eq:NLS} satisfies 
\begin{equation}\label{eq:iniE_w}
    E(\psi_0)=\int_\T \frac12|\d_x\psi_0|^2+f(|\psi_0|^2)+\frac12(\d_x V_0)^2dx \leq \frac12(M_0^\frac12-\delta M_0^{-\frac12})^2,
\end{equation}
where $M_0=\|\psi_0\|_{L^2}^2=\int_\T \mathcal{C}(x)dx$ and $V_0$ satisfies $-\d_x^2 V_0=|\psi_0|^2-\mathcal{C}(x)$. Then the Cauchy problem associated to \eqref{eq:NLS} with initial data $\psi_0$ has a unique global solution $\psi\in \mathcal{C}([0,\infty);H^2_x(\T))$. Moreover, $\inf_{t,x}|\psi|\geq \delta^\frac12$ and the hydrodynamic functions associated to $\psi$, given by
\[
\rho=|\psi|^2,\quad J=\IM(\bar\psi\d_x\psi)=\rho v,
\]
satisfy the bounds \eqref{eq:en_disp_intro} (with $E_0=E(\psi_0)$) and \eqref{eq:bd_I} on any finite time interval $[0,T]$, $T>0$.
\end{thm}
\begin{rem}
Let us remark that here and hereafter we use the same notation both for the total energy associated to \eqref{eq:QHD} and to \eqref{eq:NLS}, see \eqref{eq:en} and \eqref{eq:iniE_w}, respectively. We stress that this does not provide any ambiguity, as the two functionals coincide for Schr\"odinger-generated hydrodynamic data, namely defined through $(\rho, J)=(|\psi|^2, \IM(\bar\psi\d_x\psi))$.
\end{rem}

\subsection{Main result: relaxation-time limit}

We now turn our attention to the relaxation-time limit $\tau\to 0$. Recall the scaling \eqref{eq:rs_intro},
\begin{equation}\label{eq:rs_intro_1}
t'=\tau t,\quad (\rho_\tau,J_\tau)(t',x)=\left(\rho,\frac{1}{\tau}J\right)\left(\frac{t'}{\tau},x\right),
\end{equation}
then the QHD system \eqref{eq:QHD} can be written as it appears in \eqref{eq:QHD_rs_intro}, which formally converges towards the QDD equation \eqref{eq:qdde} as $\tau\to 0$.

The first rigorous justification was given by \cite{JLM} for small and smooth perturbations around a constant state, and later extended to the bipolar case, namely the system of electrons and ions, by \cite{LZZ}. In \cite{LT}, the relaxation-time limit is considered for finite energy weak solutions, however they need to assume the initial data to be well-prepared (see the next paragraph) and the solutions of the limiting equations \eqref{eq:qdde} to be sufficiently smooth. On the other hand, the analysis of relaxation-time limits is also performed in the classical hydrodynamic equations, see for example \cite{HL,HP,JP}.

Let us notice that, under the scaling \eqref{eq:rs_intro_1}, the solution $(\rho_\tau, J_\tau)$ to \eqref{eq:QHD_rs_intro} satisfies the following energy inequality
\begin{equation}\label{eq:B1_1_rs}
E_\tau(t')+\int_0^{t'}\frac{(J_\tau)^2}{\rho_\tau}\,dxds\leq C E_\tau(0),
\end{equation}
where now the rescaled energy reads
\begin{equation}\label{eq:en_rs_intro}
E_\tau(t')=\int_{\mathbb T}\frac12(\d_x\sqrt{\rho_\tau})^2+\frac{\tau^2}{2}\frac{(J_\tau)^2}{\rho_\tau}+f(\rho_\tau)+
\frac12(\d_xV_\tau)^2\,dx.
\end{equation}
Notice that \eqref{eq:en_rs_intro} and \eqref{eq:B1_1_rs} are the rescaled versions of \eqref{eq:en} and \eqref{eq:en_disp_intro}, respectively.
\newline
As we already mentioned, the bounds implied by \eqref{eq:B1_1_rs} are not sufficient to pass to the limit in \eqref{eq:QHD_rs_intro}. Indeed, while the uniform $L^2_{t', x}$ estimate on $J_\tau/\sqrt{\rho_\tau}$ allow us to show the convergence to zero of the inertial term, there are no sufficient bounds on $\d_x\sqrt{\rho_\tau}$ to analyze the weak limit of the quadratic term $(\d_x\sqrt{\rho_\tau})^2$. For this reason, we study the relaxation-time limit in the framework of GCP solutions, see Definition \ref{def:GCPsln} for more details, whose existence is provided by Theorem \ref{thm:glob2}. 
Analogously to \eqref{eq:B1_1_rs}, the rescaled version of \eqref{eq:bd_I} reads 
\begin{equation}\label{eq:bd_I_rs}
I_\tau(t')+\int_0^{t'}\int_{\mathbb T}(\d_{t'}\sqrt{\rho_\tau})^2+\rho_\tau v_\tau^4\,dxds\leq 
C(M_0, E_0, I_0, \delta, t'),
\end{equation}
where now $I_\tau$ is defined by
\begin{equation}\label{eq:higher_rs_intro}
I_\tau(t')=\int_{\mathbb T}\frac{\tau^2}{2}\rho_\tau\left(\mu_\tau^2+\sigma_\tau^2\right)\,dx,
\end{equation}
and the scaled chemical potential is given by 
\begin{equation}\label{eq:chem_rs_intro}
\mu_\tau=\frac{1}{\tau}\left(-\frac{\d_x^2\sqrt\rho_\tau}{2\sqrt\rho_\tau}+\frac{\tau^2}{2}v_\tau^2+f'(\rho_\tau)+V_\tau\right).
\end{equation}
For a more detailed discussion about the rescaled quantities, we refer to Section \ref{sect:rs}.

One important question that should be noticed in the relaxation-time limit is the phenomenon of initial layer. For the rescaled QHD system \eqref{eq:QHD_rs_intro}, the initial data consists of two components, namely the initial density $\rho_0$ and the initial momentum density $J_{\tau,0}$. On the other hand, the initial data for the QDD equation \eqref{eq:qdde} only contains the initial density $\rho_0$, and the limiting momentum density is required to satisfy the constitutive relation
\begin{equation}\label{eq:barJ}
\bar{J}=\frac{1}{2}\bar\rho\d_x\left(\frac{\d_x^2\sqrt{\bar\rho}}{\sqrt{\bar\rho}}\right)-\d_x p(\bar\rho)-\bar\rho\d_x\bar{V},
\end{equation}
which can be formally deduced by passing to the limit $\tau\to 0$ in the second equation of \eqref{eq:QHD_rs_intro}. 
The initial data for system \eqref{eq:QHD_rs_intro} are  well-prepared if $J_{\tau, 0}=\bar J_0$, where $\bar J_0$ is defined by $\bar\rho=\rho_0$ in \eqref{eq:barJ}.
In our analysis we do not assume the initial data to be well-prepared, therefore an initial layer will naturally appear. As a consequence, the convergence of the momentum density will be only in $\{t>0\}$.
\newline
%One of the purposes of this paper is to analyze the relaxation-time limit for a large class of weak solutions of \eqref{eq:QHD}, \textcolor{red}{and we will not assume the initial momentum to be well-prepared}. Moreover, a
Another novelty with respect to the existing literature is that we are able to provide an explicit rate of convergence with respect to $\tau$.

\begin{thm}[Relaxation-time limit]\label{thm:rlxlimit}
Let $\{(\rho_{\tau},J_{\tau})\}_{\tau>0}$ be a sequence of GCP solutions of the rescaled equation \eqref{eq:QHD_rs_intro}, such that
\begin{itemize}
\item $\inf_{t’,x}\rho_{\tau}\geq \delta$ for a uniform constant $\delta>0$,
\item $(\rho_{\tau},J_{\tau})$ satisfy the bounds \eqref{eq:B1_1_rs} and \eqref{eq:bd_I_rs}.
\end{itemize}
Then the relaxation-time limit holds, namely $\{\rho_{\tau}\}$ converges to a limiting density $\bar\rho$ in the form
\[
\|\rho_\tau-\bar\rho\|_{L^\infty_{t'}L^2_x}+\|\sqrt\rho_\tau (\d_x^2\log\sqrt\rho_\tau-\d_x^2\log\sqrt{\bar\rho})\|_{L^2_{t',x}}\leq C(M_0,E_0,I_0,\delta,T)\tau,
\]
for any $0<T<\infty$. Moreover, $\bar\rho$ is a weak solution of \eqref{eq:qdde} in the sense of Definition \ref{def:qdde_ws}. 

Finally, if the internal energy $f(s)$ is convex, then $\bar\rho$ is a dissipative solution in the sense of Definition \ref{def:qdde_ws}.

\end{thm}

We remark that our result works for general initial data $(\rho_0,J_0)$, and we allow the existence of initial layer. On the other hand, a better convergence rate can be obtained by using the relative energy method and in addition by assuming well-prepared initial data, see for instance \cite{LT}.

The proof of Theorem \ref{thm:rlxlimit} relies on deriving suitable a priori bounds for solutions to the hydrodynamical system, in order to pass to the limit (at least in the distributional sense) all terms in equations \eqref{eq:QHD_rs_intro}. Similar results are already present in the literature, especially for viscous hydrodynamic models. For instance, the authors in \cite{Blub} consider a one-dimensional Navier-Stokes-Korteweg-type system with (degenerate) viscosity, linear capillarity and nonlinear damping. By exploiting the a priori bounds derived from energy and BD-entropy estimates, the authors prove convergence in the class of weak solutions. More precisely, the BD entropy estimate yields additional regularity properties for the mass density, providing the sufficient compactness and allowing to pass to the limit even in the higher order nonlinear terms. Moreover, it is also shown in \cite{Blub} that the (scaled) BD entropy converges, in the limit $\tau\to0$, to the so called Bernis-Friedman entropy estimate for the limiting gradient flow equation. See also \cite{ACLS}, where the relaxation-time limit can be obtained (with no rate) for general weak solutions to quantum Navier-Stokes equations in the three-dimensional torus.

In the absence of viscous terms, such as for \eqref{eq:QHD}, no extra Sobolev bounds may be obtained from entropy estimates, such as of BD-type. 
%obtaining sufficient a priori bounds becomes more difficult. 
In particular, the sole energy bounds cannot ensure the convergence of the quadratic term $(\d_x\sqrt\rho_\tau)^2$ appearing in the quantum stress tensor, see \eqref{eq:bohm}. To overcome this problem, and also to obtain dissipative properties of the rescaled equation \eqref{eq:QHD_rs_intro}, we introduce the physical entropy
\[
H(\rho_\tau)=\int_\T\rho_\tau\log\left(\frac{\rho_\tau}{M_0}\right)dx=\int_\T g(\rho_\tau)dx,
\]
where $M_0=\int_\T\rho_\tau dx$ is the conserved total mass of $\rho_\tau$, and $g(s)=s\log(s/M_0)$. 
The functional $H(\rho_\tau)$ is well-defined and can be uniformly bounded in time for GCP solutions, see Section \ref{sect:entrp}.
These estimates provide us with the $L^2_{t'}H^2_x$ bound of $\sqrt\rho_\tau$ uniformly with respect to $\tau$, which allows us to prove the relaxation-time limit. Moreover, we obtain an explicit convergence rate of this limit with respect to  $\tau$ by considering the relative entropy 
\begin{align*}
H(\rho_\tau|\bar\rho)(t')=\int_\T g(\rho_\tau)-g(\bar\rho)-g'(\bar\rho)(\rho_\tau-\bar\rho)dx,
\end{align*}
where $\bar\rho$ is the limiting solution of the QDD equation \eqref{eq:qdde}.

In fact the estimates given by the energy, the GCP fucntional and the entropy also allow us to prove a time decay estimate of $(\rho_\tau,J_\tau)$. By using \eqref{eq:bohm}, in the special case when there is no pressure and no electrostatic potential, the QDD equation \eqref{eq:qdde} reduces to the following fourth-order parabolic equation,
\begin{equation}\label{eq:DLSS}
\d_{t'}\bar\rho+\frac14\d_x^2(\bar\rho\d_x^2\log\bar\rho)=0,
\end{equation}
which independently appears also in various places of the  mathematical physics literature. Equation \eqref{eq:DLSS} has been first derived by Derrida, Lebowitz, Speer, and Spohn in \cite{DLSS1, DLSS2}, and we therefore refer to it as the \emph{DLSS equation}. A similar equation was also obtained in  \cite{Ber, BP1} as a model for thin films. Concerning the mathematical analysis of the DLSS equation, we address to \cite{DGJ,GST,JM} and references therein. By using different techniques, they prove the existence of non-negative solutions and show that the entropy computed along strictly positive solutions decays exponentially in time due to the dissipative feature of fourth-order parabolic equation. 
\newline
This latter property is actually true also for the hydrodynamic model. In particular, in the present paper we show that also solutions of the rescaled collisional QHD system \eqref{eq:QHD_rs_intro} possess exponentially decaying physical entropy (Theorem \ref{thm:disp}) for $\tau$ small enough and suitable functions $f(\rho)$ and $\mathcal{C}(x)$.  More precisely, in the following Theorem we assume the internal energy to be of the form $f(\rho_\tau)=(\rho_\tau-M_0)^{2n}$, $n\in\N$, and the doping profile to be $\mathcal{C}(x)=M_0$. This property distinguishes the collisional QHD system ($\frac{1}{\tau}>0$) from the non-collisional ($\frac{1}{\tau}=0$) one, and eventually breaks the time invertible structure of quantum mechanics as the strength of the thermal interaction between particles goes to infinity.

\begin{thm}[Dissipation for small $\tau$]\label{thm:disp}
Let $(\rho_\tau,J_\tau)$ be a GCP solution of the rescaled QHD system \eqref{eq:QHD_rs_intro} with $\inf_{t',x}\rho\geq\delta>0$. We also assume the internal energy to be $f(\rho_\tau)=(\rho_\tau-M_0)^{2n}$, $n\in\N$, and the electric background density $\mathcal{C}(x)\equiv M_0$. Then, there exists a constant $C=C(M_0, E_0, \delta)$, 
%depending on the initial total mass $M_0$, the initial energy $E_0$ and $\delta$, 
such that, by denoting $c_1 =C(M_0,E_0,\delta)^{-1}$, if $\tau\leq \tau_0=\min\left\{c_1, \frac{c_1^\frac12}{4},\left(\frac{2c_1\delta}{8+\delta}\right)^\frac12\right\}$, then the functional 
\[
F_\tau(t')=H(\rho_\tau)+E_\tau(t')+c_1 I_\tau(t')
\]
satisfies
\begin{equation}
F_\tau(t')\leq F_\tau(0)\exp\left(C(M_0,E_0, \delta)\int_0^{t'}\int_\T\rho_\tau v_\tau^2dxds-t'\right).
\end{equation}
In particular, $H(\rho_\tau)$, $E_\tau(t')$and $I_\tau(t')$ decay exponentially for $t'$ large.
\end{thm}

\begin{rem}
For the DLSS equation \eqref{eq:DLSS}, the authors of \cite{JM} provide the following example of a non-dissipative solution,
\[
\rho_\tau=\cos^2x,\quad J_\tau=0,
\]
which is also a solution to equation \eqref{eq:QHD_rs_intro} in the pressureless and zero-electric field case ($p(\rho)=0$ and $V=0$). This example suggests that the dissipation depends on the strictly positivity of $\rho$.
\end{rem}

\subsection{Outline of the paper}

Our paper is structured as follows. Definitions and preliminaries results are given in Section 2. In Section \ref{sect:exist} we prove   Theorem \ref{thm:glob2} by analyzing the Schr\"odinger-Langevin equation \eqref{eq:NLS} and the a priori bounds given by $E(t)$ and $I(t)$. The relation between \eqref{eq:NLS} and \eqref{eq:QHD} is rigorously established by the methods of wave function lifting and polar factorization developed in \cite{AM1,AM2,AMZ1}. Starting from Section \ref{sect:rs}, we consider the rescaled system \eqref{eq:QHD_rs_intro}. The property of the physical entropy $H(\rho_\tau)$ and dissipation for GCP solutions with small $\tau$ are interpreted in Section \ref{sect:entrp}. Finally, the rigorous proof of the relaxation-time limit in the framework of GCP solutions is concluded in Section \ref{sect:rxl}.

\section{Definitions and preliminaries}\label{sect:def}
In this section we first provide some notation and definition that will be used through out this paper.

Lebesgue and Sobolev norms on $\T$ are denoted by
\[
||f||_{L_{x}^{p}}\coloneqq(\int_{\T}|f(x)|^{p}dx)^{\frac{1}{p}},
\]
\[
||f||_{W_{x}^{k,p}}\coloneqq\sum_{j=0}^k||\partial_{x}^{j}f||_{L_{x}^{p}},
\]
and $H_{x}^{k}\coloneqq H^{k}(\T)$ denotes the Sobolev space $W^{k,2}(\T)$.
Given a time interval $I\subset[0, \infty)$, the mixed space-time Lebesgue norm of $f:I\to L^{r}(\T)$ is
given by
\[
||f||_{L_{t}^{q}L_{x}^{r}}\coloneqq\left(\int_{I}||f(t)||_{L_{x}^{r}}^{q}dt\right)^{\frac{1}{q}}=\left(\int_{I}\left(\int_{\T}|f(x)|^{r}dx\right)^{\frac{q}{r}}dt\right)^{\frac{1}{q}}.
\]
Similarly the mixed
Sobolev norm $L_{t}^{q}W_{x}^{k,r}$ is defined. We use $C$ to denote a generic constant,$C(A)$ indicates its dependence on the quantity $A$.

By using identity \eqref{eq:bohm} and defining the velocity $v=\frac{J}{\rho}$, system \eqref{eq:QHD} becomes
\begin{equation}\label{eq:QHD_1d}
\left\{\begin{aligned}
&\d_t\rho+\d_x(\rho v)=0\\
&\d_t(\rho v)+\d_x[\rho v^2+p(\rho)+(\d_x\sqrt{\rho})^2]+\rho\d_xV=\frac14\d_{x}^3\rho-\frac{1}{\tau}\rho v.
\end{aligned}\right.
\end{equation}
Let us also recall the total mass and total energy are respectively given by
\begin{equation*}
M(t)=\int_\T\rho(t, x)\,dx,
\end{equation*}
and
\begin{equation}\label{eq:en_1d}
E(t)=\int_{\T}\frac12\rho v^2+\frac12(\d_x\sqrt{\rho})^2+f(\rho)+\frac12(\d_xV)^2dx.
\end{equation}

As already mentioned in the Introduction, there is a formal analogy between system \eqref{eq:QHD_1d} and the NLS-type equation \eqref{eq:NLS}, through the so-called Madelung transformation \cite{Mad}. Indeed, this may be immediately seen by expressing the wave function $\psi$ in terms of its amplitude and phase, i.e.
\begin{equation}\label{eq:mad}
\psi=\sqrt{\rho}e^{iS}.
\end{equation}
By plugging this ansatz into \eqref{eq:NLS} and by further separating real and imaginary parts of the resulting equation, we deduce a system of evolution equations for $\rho$ and $S$, 
\begin{equation}\label{eq:qHJ}
\left\{\begin{aligned}
&\d_t\rho+\d_x(\rho\nabla S)=0\\
&\d_tS+\frac12|\d_x S|^2+f'(\rho)+V+\frac1\tau S=\frac12\frac{\d_{xx}\sqrt{\rho}}{\sqrt{\rho}}.
\end{aligned}\right.
\end{equation}
If we now differentiate the second equation in \eqref{eq:qHJ} with respect to the space variable, we obtain an Euler-type equation for the velocity field $v=\d_xS$, namely
\begin{equation}\label{eq:qvel}
\d_tv+\d_x\left(\frac12v^2+f'(\rho)+V-\frac12\frac{\d_{xx}\sqrt{\rho}}{\sqrt{\rho}}\right)=-\frac1\tau v.
\end{equation}
By recalling the definition of the chemical potential in \eqref{eq:chem}, we see that \eqref{eq:qvel} equals the second equation in \eqref{eq:Ham_form}.
The momentum density equation in the QHD system \eqref{eq:QHD} is then derived by multiplying \eqref{eq:qvel} by $\rho$ and by using the continuity equation.
\newline
In general, the analogy between \eqref{eq:NLS} and \eqref{eq:QHD_1d} fails when the Madelung transformation \eqref{eq:mad} becomes singular, namely in the presence of vacuum regions. It is, however, possible to provide a rigorous justification of such equivalence by means of the polar factorization and wave function lifting techniques, developed by the authors in \cite{AM1, AM2, AMZ1}. In the present paper, since we focus on vacuum-free solutions, we will omit some technical details related to the polar factorization and wave function lifting, and we simply provide a quick overview of the main objects. For a more comprehensive discussion, we address the interested reader to \cite{AM_b, HaoMinE}.
\newline
First,
it is straightforward to see that the objects $\sqrt{\rho}$ and $\sqrt{\rho}v$ provide a senseful definition of weak solutions to \eqref{eq:QHD_1d} and the energy bound in \eqref{eq:en_1d}. Namely, by denoting $\Lambda=\sqrt{\rho}v$, the \emph{hydrodynamic state} $(\sqrt{\rho}, \Lambda)$ carries all the necessary information.

\begin{defn}[Finite energy weak solutions]\label{def:FEWS}
Let $\rho_0, J_0\in L^1(\T)$ and $\rho_0>0$. We say the pair $(\rho, J)$ is a weak solution to the Cauchy problem for \eqref{eq:QHD_1d} with initial data
$
\rho(0)=\rho_0$, $(\rho v)(0)=J_0,
$
on $[0, T)\times\T$ if there exists $(\sqrt{\rho}, \Lambda)$ such that $\rho=(\sqrt{\rho}^2)$, $J=\sqrt{\rho}\Lambda$ and $(\rho, J)$ satisfies the QHD system \eqref{eq:QHD_1d} in the sense of distribution.
Moreover, we say $(\rho,J)$ is a finite mass weak solution, if for almost every $t\in[0, T)$ we have $M(t)<\infty$. Analogously, $(\rho,J)$ is called a finite energy weak solution if for almost every $t\in[0, T)$, we also have $E(t)<\infty$. 
\end{defn}
\begin{rem}
Notice that in the present paper we deal with vacuum-free solutions, hence the equation \eqref{eq:QHD_1d} may be expressed in terms of $\rho$ and $v$.
\end{rem}
Moreover, we recall the polar factorization Lemma, first developed in Section 3 of \cite{AM1,AM2}, that defined the hydrodynamic state $(\sqrt{\rho}, \Lambda)$ associated to a given wave function $\psi\in H^1$, with no restriction on the vacuum region. This is achieved by exploiting the polar factor $\phi$ associated to $\psi$, defined through the identity $\psi=|\psi|\phi$. A proof of the following Lemma may be found in \cite[Proposition 1.2]{AM_b}.
\begin{lem}[Polar factorization]\label{lem:polar}
Let $\psi\in H^1(\T)$ and let $P(\psi)$ be the  set of polar factors, namely
\[
P(\psi)=\{\phi\in L^\infty(\T);\,\psi=|\psi|\phi\quad\textrm{for a.e }x\in\T\}.
\]
Then for $\phi\in P(\psi)$, by defining $\sqrt{\rho}=|\psi|$, $\Lambda=\IM(\bar\phi\d_x\psi)$, we have:
\begin{itemize}
\item $\sqrt{\rho}\in H^1$ and it holds 
$\d_x\sqrt{\rho}=\RE(\bar\phi\d_x\psi)$;
\item 
$|\d_x\psi|^2=
(\d_x\sqrt{\rho})^2+\Lambda^2$ a.e. in $\T$.
\end{itemize}
\end{lem}

We also recall the definition of GCP solutions, characterized by the chemical potential introduced in \eqref{eq:chem}
\[
\mu=-\frac{\d_x^2\sqrt\rho}{2\sqrt\rho}+\frac12v^2+f'(\rho)+V
\]
and its $L^2$ norm with respect to density measure $\rho dx$.
As before, we remark that this definition may comprise a more general class of weak solutions, allowing for states with vacuum regions. We refer to \cite{AMZ1} for more details.
In what follows we provide the definition of weak solutions with bounded generalized chemical potential.

\begin{defn}[GCP solutions]\label{def:GCPsln}
Let $(\rho, v)$ be a finite energy weak solution to \eqref{eq:QHD_1d} on 
$[0,T]\times\R$. We say that $(\rho, v)$ is a GCP solution for the system \eqref{eq:QHD_1d} if the following bounds are satisfied
\begin{equation*}
\begin{aligned}
\|\sqrt{\rho}\|_{L^\infty(0, T; H^1(\T))}+\|\sqrt\rho v\|_{L^\infty(0, T; L^2(\T))}\leq& C,\\
\|\d_t\sqrt{\rho}\|_{L^\infty(0, T; L^2(\T))}+\|\sqrt\rho\mu\|_{L^\infty(0, T; L^2(\T))}\leq& C.
\end{aligned}
\end{equation*}
\end{defn}

%As discussed in Section \ref{sect:intro}, for $(\rho,v)$ generated by a wave function $\psi$ such that $|\psi|\geq \delta >0$, we have
%\[
%v=\IM\left(\frac{\d_x\psi}{\psi}\right),\quad\mu=-\frac12\RE\left(\frac{\d_x^2\psi}{\psi}\right)+f'(|\psi|^2)+V.
%\]
%Thus by WKB ansatz, the phase function $S$, formally given by $S=\frac{1}{i}\log\left(\frac{\psi}{|\psi|}\right)$, of $\psi$ should satisfy a space-time gradient equation 
Let us recall that, by the Madelung transformation \eqref{eq:mad}, we have $v=\d_xS$. Consequently, from \eqref{eq:qHJ}, we have
\begin{equation}\label{eq:gradS}
\begin{cases}
\d_xS =v,\quad\d_t S=-\mu-\frac{1}{\tau}S \\
S(0,0)=S_*.
\end{cases}
\end{equation}
On the other hand, the equation \eqref{eq:qvel} for the velocity field may be interpreted as the "irrotationality" condition for the gradient equation \eqref{eq:gradS}, namely
\begin{equation}\label{eq:qvel_irrot}
\d_tv+\d_x(\mu+\frac1\tau S)=0.
\end{equation}
Thus we can formally write the phase function $S$ in hydrodynamic functions as
\begin{equation}\label{eq:defS_pt}
S(t,x)=e^{-\frac{t}{\tau}} \left(S_*+\int_0^{x_*}v_0(y)dy\right)-\int_0^t e^{\frac{s-t}{\tau}}\mu(s,x_*)ds+\int_{x_*}^x v(t,y)dy,
\end{equation}
where $x_*\in \T$ is an arbitrary point and has no influence on the definition due to the "irrotationality" condition \eqref{eq:qvel_irrot}. However, for weak solutions $(\rho,J)$, the chemical potential $\mu$ is merely a Lebesgue function or even a distribution. Since \eqref{eq:defS_pt} should be independent of the point $x_*$, we can take the average of \eqref{eq:defS_pt} with respect to $x_*$, which allows us to extend the hydrodynamic definition of the phase function $S$ to more general cases.

\begin{defn}[Phase function]\label{def:phase}
Let $(\rho,v)$ be a pair a solution of \eqref{eq:QHD} on $[0,T)\times\T$ such that the velocity $v$ and the chemical potential $\mu$ given by \eqref{eq:chem} satisfy, 
\[
v\in L^\infty_tL^1_x,\quad\mu\in \mathcal{D}([0,T)\times\T),
\]
and $\int_\T v dx=0$. Then we define the phase function $S$ to be the solution of the gradient equation \eqref{eq:gradS}, which can be explicitly written as
\begin{equation}\label{eq:defS}
\begin{aligned}
S(t,x)=&e^{-\frac{t}{\tau}} \left(S_*+\int_\T\int_0^{x_*}v_0(y)dydx_*\right)\\
&-\int_0^t \int_\T e^{\frac{s-t}{\tau}}\mu(s,x_*)dx_*ds+\int_\T\int_{x_*}^x v(t,y)dydx_*.
\end{aligned}
\end{equation}
\end{defn}

\begin{rem}
The condition $\int_\T v dx=0$ is needed to make $S$ a periodic function on $\T$, however Definition \ref{def:phase} still holds without this condition.
\end{rem}

\begin{rem}\label{rem:equivS}
Let us remark that the assumption $|\psi|\geq\sqrt{\delta}$ clearly implies that the logarithm $S=\frac{1}{i}\log\left(\frac{\psi}{|\psi|}\right)$ is well defined. Consequently, by the uniqueness of equation \eqref{eq:gradS}, it coincides with \eqref{eq:defS}, if we suitably choose $S_*$.
%In the case that the range of $\psi$ is restricted in a simply connected complex domain strictly away from $0$, which is the case we consider in Theorem \ref{thm:glob2}, the logarithm formulation $S=\frac{1}{i}\log\left(\frac{\psi}{|\psi|}\right)$ is well defined. Then by the uniqueness of equation \eqref{eq:gradS}, the logarithm formulation coincides with \eqref{eq:defS} if we choose $S_*$ suitably.
\end{rem}

Next, the physical entropy $H(\rho)$ of a density function $\rho\geq 0$ on $[0,T)\times\T$ is defined by
\begin{equation*}
H(\rho)=\int_\T\rho\log\left(\frac{\rho}{M_0}\right),
\end{equation*}
where $M_0\equiv M(t)$ is the conserved total mass of $\rho$ given by \eqref{eq:mass}. Since we assume $|\T|=1$, $M_0$ can also be viewed as the average of $\rho$, therefore by convexity of the function $g(s)=s\log s$, we have 
\[
H(\rho)=\int_\T \rho\log\rho dx-M_0\log M_0\geq 0.
\]
For later use in the analysis of relaxation-time limit, we now give the definition of weak solutions to the quantum drift-diffusion equation \eqref{eq:qdde}.

\begin{defn}\label{def:qdde_ws}
We say $\bar\rho$ is a weak solution to the quantum drift-diffusion equation \eqref{eq:qdde} with initial data $\bar\rho(0)=\bar\rho_0\in L^1(\T)$ on $[0,T)\times\T$, if $\sqrt{\bar\rho}\in L^2_{loc}(0,T;H^1(\T))$ such that for any $\eta\in \mathcal{C}_0^\infty([0,T)\times\T)$,
\begin{equation}\label{eq:qdde_cty}
\begin{aligned}
\int_0^T\int_\T\bar\rho\d_t\eta dxdt+\int_0^T\int_\T&[(\d_x\sqrt{\bar\rho})^2+p(\bar\rho)]\d_x^2\eta-\bar\rho\d_x \bar V\d_x\eta\\
&-\frac14\bar\rho\d_x^4\eta dxdt+\int_\T\bar\rho_0\eta(0)dx=0.
\end{aligned}
\end{equation}

Moreover, $\bar\rho$ is called a dissipative solution if $\d_x^2\sqrt{\bar\rho}\in L^2([0,T)\times\T)$ and $\d_x(\bar\rho^\frac14)\in L^4([0,T)\times\T)$, which satisfies
\begin{equation}\label{eq:disp_entrp_1}
H(\bar\rho)(t)+\int_0^t\int_\T(\d_x^2\sqrt{\bar\rho})^2+(\d_x(\bar\rho^\frac14))^4dxdt\leq H(\rho_0)
\end{equation}
for any $0\leq t<T$.
\end{defn}

For later use of this paper, we collect some preliminary results that provide a rigorous connection between the collisional QHD system \eqref{eq:QHD} and the NLS equation \eqref{eq:NLS}. More specifically, we are going to recall the wave function lifting method, that gives sufficient and necessary conditions for a hydrodynamic data $(\rho,v)$ to have an associated wave function $\psi\in H^1(\T)$. 
The wave function lifting was first introduced in \cite{AMZ1,AMZ2}, where it also applies to general cases where non-trivial vacuum regions are allowed. In this paper we may restrict to the simpler case of non-vanishing wave functions.
The following Proposition collects some useful facts about the $H^1$ wave function lifting, for more details we refer to \cite[Proposition 17 and Lemma 19]{AMZ1} and their proofs.

\begin{prop}[$H^1$ wave function lifting]\label{prop:lift1}
Let $(\rho,v)$ be a hydrodynamic data with strictly positive density $\inf_x\rho\geq \delta >0$. There exists a unique (upto constant phase shifts) associated wave function $\psi\in H^1(\T)$ 
%in the sense of \eqref{eq:polar},
if and only if there exists a constant $0<M_1<\infty$ such that 
$$
\|\sqrt{\rho}\|_{H^1(\T)}+\|\sqrt\rho v\|_{L^2(\T)}\leq M_1.
$$ 
Furthermore we have 
\[
\d_x\psi=(\d_x\sqrt\rho+i\rho v)\phi,\quad \|\psi\|_{H^1(\T)}\leq C(M_1),
\]
where $\phi\in P(\psi)$ is a polar factor of $\psi$.
\end{prop}

The $H^2$ wave function lifting is given in the framework of the generalized chemical, see Section 3 in \cite{AMZ1} for a more detailed discussion. A more complete result, valid also for wave functions with non-trivial vacuum regions, may be found in \cite[Proposition 24]{AMZ1}.
%In the original results in \cite{AMZ1}, vacuum regions are allowed provided the continuity of energy density at vacuum boundaries. However, in this paper we only need a simplified version where density is strictly positive.

\begin{prop}[$H^2$ wave function lifting]\label{prop:lift2}
Let $(\rho,v)$ be a hydrodynamic data with strictly positive density $\inf_x\rho\geq \delta >0$. Let us also assume 
\begin{equation}\label{eq:bd_lift2}
\begin{aligned}
&\|\sqrt{\rho}\|_{H^1}+\|\sqrt\rho v\|_{L^2}\leq M_1,\\
&\|\d_x (\rho v)/\sqrt{\rho}\|_{L^2(\T)}+\|\sqrt{\rho} v^2-\d_x^2\sqrt{\rho}\|_{L^2(\T)}\leq M_2,
\end{aligned}
\end{equation}
for some constants $M_1,M_2<\infty$. Then there exists a unique wave function $\psi\in H^2(\T)$, associated to $(\sqrt\rho,\Lambda)$, and it follows that
\begin{equation}\label{eq:H2_lift}
\|\psi\|_{H^2(\T)}\leq C(M_1,M_2).
\end{equation}
\end{prop}

\section{Global well-posedness of GCP solutions}\label{sect:exist}

The global well-posedness Theorem \ref{thm:glob2} of GCP solution to system \eqref{eq:QHD} is proved analysing the existence $H^2$ solutions to the NLS \eqref{eq:NLS} and establishing the relationship between \eqref{eq:QHD} and \eqref{eq:NLS} by Proposition \ref{prop:lift2}.
%and Proposition \ref{prop:polar}.

In the first part of this section, we will focus on the Cauchy problem of NLS \eqref{eq:NLS},
\begin{equation}\label{eq:cauchy_NLS}
\begin{cases}
i\d_t\psi+\frac12\d_x^2\psi=f'(|\psi|^2)\psi+\frac{1}{\tau}S\psi+V\psi\\
-\d_x^2 V=|\psi|^2-\mathcal{C}(x),\quad \psi(0)=\psi_0
\end{cases}
\end{equation}
for initial data $\psi_0\in H^2(\T)$, where $S$ is the phase function of $\psi$ given in Definition \ref{def:phase}. We first prove the local existence of $H^2$ solutions to \eqref{eq:cauchy_NLS} by a standard Picard iteration scheme.

\begin{prop}
Let us assume the initial data $\psi_0\in H^2(\T)$ satisfy $\inf_x|\psi_0|\geq \delta^\frac12 >0$ and 
\begin{equation}\label{eq:ini_1}
\|\d_x\psi_0\|_{L^2_x} \leq M_0^\frac12-\delta M_0^{-\frac12},
\end{equation}
where $M_0=\|\psi_0\|_{L^2_x}^2$. Then the Cauchy problem \eqref{eq:cauchy_NLS} has a unique solution $\psi\in L^\infty(0,T_*;H^2(\T))$ local in time, for some $T_*>0$, with $\inf_{t,x}|\psi|\geq \frac{\delta^\frac12}{2}$ .
\end{prop}

\begin{proof}
The proof follows a standard Picard iteration argument. The sequence $\{\psi_n\}_{n\geq 1}$ for \eqref{eq:cauchy_NLS} is defined iteratively by solving $\psi_n$ from the linear Schr\"odinger equation 
\begin{equation}\label{eq:QLS_n}
i\d_t\psi_n+\frac12\d_x^2\psi_n=f'(|\psi_{n-1}|^2)\psi_n+\frac{1}{\tau}S_{n-1}\psi_n+V_{n-1}\psi_n
\end{equation}
on $[0,T_*)\times \T$, for small $T_*$ to be determined later, with $\psi_{n-1}$, $S_{n-1}$ and $V_{n-1}$ given in the previous step. The Duhamel formula for \eqref{eq:QLS_n} is written as
\begin{equation}\label{eq:duhamel}
\begin{aligned}
\psi_n=e^{\frac{i}{2}t\d_x^2}\psi_0&-i\int_0^te^{\frac{i}{2}(t-s)\d_x^2}[f'(|\psi_{n-1}|^2)+V_{n-1}]\psi_n(s)ds\\
&-\frac{i}{\tau}\int_0^te^{\frac{i}{2}(t-s)\d_x^2}S_{n-1}\psi_n(s)ds,
\end{aligned}
\end{equation}
and the electric potential $V_{n-1}$ is provided by the Poisson equation, 
\[
-\d_x^2 V_{n-1}=|\psi_{n-1}|^2-\mathcal{C}(x)
\]
with $\int_\T V_{n-1}dx=0$. By Poincar\'e inequality, $V_{n-1}$ is bounded by
\[
\|V_{n-1}\|_{W^{1,p}_x}\leq C(\|\psi_{n-1}\|^2_{L^2_x}+\|\mathcal{C}\|_{L^1_x}),\quad 1\leq p\leq \infty.
\]
For the phase function, we first define $S_0=0$. For $n\geq 1$, we want to give s suitable definition of $S_n$ by adapting the idea of Definition \ref{def:phase} for $\psi_n$ as follows. As before, we define the velocity and chemical potential as
\begin{equation}\label{eq:mu_n}
v_n=\IM\left(\frac{\d_x\psi_n}{\psi_n}\right),\quad \mu_n=-\frac12\RE\left(\frac{\d_x^2\psi_n}{\psi_n}\right)+f'(|\psi_{n-1}|^2)+V_{n-1}.
\end{equation}
and we will see below this definition is consistent by the positivity of $|\psi_n|$. By \eqref{eq:QLS_n} the equation of $v_n$ is given by
\begin{equation}\label{eq:dtv_n}
\d_t v_n=-\d_x\left(\mu_n+\frac{1}{\tau}S_{n-1}\right).
\end{equation}
Thus following the idea of \eqref{eq:gradS}, we define $S_n$ to be the solution to the space-time gradient equation 
\begin{equation}\label{eq:gradS_n}
\begin{cases}
\d_xS_n =v_n,\quad\d_t S_n=-\mu_n-\frac{1}{\tau}S_{n-1} \\
S_n(0,0)=S_*\in [0,2\pi).
\end{cases}
\end{equation}
The "irrotationality" condition \eqref{eq:dtv_n} guarantees the solvability of \eqref{eq:gradS_n}, and following the idea of Definition \ref{def:phase}, the phase function $S_n$ can be explicitly written as
\begin{equation}\label{eq:defS_n}
\begin{aligned}
S_n(t,x)=&S_*+\int_\T\int_0^{x_*}v_0(y)dydx_*\\
&-\int_0^t \int_\T(\mu_n+\tau^{-1}S_{n-1})(s,x_*)dx_*ds+\int_\T\int_{x_*}^x v_n(t,y)dydx_*.
\end{aligned}
\end{equation}

We claim that the sequence $\{\psi_n\}$ satisfies the following properties on $[0,T_*)\times \T$:
\begin{itemize}
\item[(1)] boundedness: $\|\psi_n\|_{L^\infty_tH^2_x}\leq 2\|\psi_0\|_{H^2_x}$ and $\|S_{n}\|_{L^\infty_tH^2_x}\leq C$;
\item[(2)] positivity: $\inf_{t,x}|\psi_n|\geq \frac{\delta^\frac12}{2}$.
\end{itemize}
Here $0<C<\infty$ may depend on $\delta$ and $\|\psi_0\|_{H^2_x}$, but is independent of $n$. By our assumption, $\psi_0$ and $S_0$ obviously satisfy the properties (1) and (2). In this paper we will only focus on the estimates related to the phase function $S_n$ and the positivity property (2), and the remaining analysis following the standard argument for non-linear Schr\"odinger equation.

For the boundedness of $\|\psi_n\|_{L^\infty_tH^2_x}$, we use the Duhamel formula \eqref{eq:duhamel} and consider its $L^\infty_tH^2_x$ norm. The last integral in \eqref{eq:duhamel} can be estimated as
\begin{align*}
\|\int_0^te^{\frac{i}{2}(t-s)\d_x^2}S_{n-1}\psi_n(s)ds\|_{L^\infty_tH^2_x}\leq&  T_* \|S_{n-1}\psi_n \|_{L^\infty_tH^2_x}\\
\leq & T_* \|S_{n-1}\|_{L^\infty_tH^2_x}\|\psi_n \|_{L^\infty_tH^2_x}.
\end{align*}
Thus by the assumption of induction, we have
\[
\frac{1}{\tau}\|\int_0^te^{\frac{i}{2}(t-s)\d_x^2}S_{n-1}\psi_n(s)ds\|_{L^\infty_tH^2_x}\leq \frac{CT_*}{\tau} \|\psi_n \|_{L^\infty_tH^2_x},
\]
and we choose $T_*$ small such that $\frac{CT_*}{\tau}\leq \frac14$. Similarly, by Sobolev inequalities, the second integral in the right hand side of \eqref{eq:duhamel} can be bounded by
\begin{align*}
\|\int_0^te^{\frac{i}{2}(t-s)\d_x^2}[f'(|\psi_{n-1}|^2)&+V_{n-1}]\psi_n(s)ds\|_{L^\infty_tH^2_x}\\
\leq & T_*\|f'(|\psi_{n-1}|^2)+V_{n-1}\|_{L^\infty_{t,x}}\|\psi_n\|_{L_t^\infty H^2_x}\\
\leq & CT_*\|\psi_n\|_{L_t^\infty H^2_x}\leq \frac14\|\psi_n\|_{L_t^\infty H^2_x},
\end{align*}
which implies
\[
\|\psi_n\|_{L_t^\infty H^2_x}\leq \|\psi_0\|_{H^2_x}+\frac12\|\psi_n\|_{L_t^\infty H^2_x},
\]
namely the first part of property (1). 

Now we prove the positivity property (2). Let us define the density $\rho_n=|\psi_n|^2$, then from \eqref{eq:QLS_n} we obtain the equation of $\rho_n$ as
\[
\d_t\rho_n+\d_x (\rho_n v_n)=0.
\]

As a consequence, the total mass (and the average density since we assume $|\T|=1$) $M_n(t)=\int_\T\rho_n dx\equiv M_0$ is conserved. By Poincar\'e inequality,
\begin{equation}\label{eq:poincare}
\|\rho_n-M_0\|_{L^\infty_{t,x}}\leq \frac12\|\d_x\rho_n\|_{L^\infty_tL^1_x}\leq M_0^\frac12\|\d_x\psi_n\|_{L^\infty_tL^2_x},
\end{equation}
therefore 
\[
\inf_{t,x}|\psi_n|^2=\inf_{t,x}\rho_n\geq \frac{\delta}{4}
\]
if we have
\begin{equation}\label{eq:bd_pc}
\|\d_x\psi_n\|_{L^\infty_tL^2_x}\leq M_0^\frac12-\frac{\delta}{4}M_0^{-\frac12}.
\end{equation}
Then we consider the norm $\|\d_x\psi_n\|_{L^\infty_tL^2_x}$. Multiplying \eqref{eq:QLS_n} by $2\d_x^2\bar\psi_n$ and taking the imaginary part, it follows that
\begin{align*}
\frac{d}{dt}\int_\T|\d_x\psi_n|^2dx=&2\int_\T[f'(|\psi_{n-1}|^2)+\tau^{-1}S_{n-1}+V_{n-1}]\IM(\psi_n\d_x^2\bar\psi_n)dx\\
\leq & 2\|[f'(|\psi_{n-1}|^2)+\tau^{-1}S_{n-1}+V_{n-1}]\|_{L^\infty_{t,x}}\|\psi_n\|_{L^\infty_{t,x}}\|\d_x^2\psi_n\|_{L^\infty_tL^2_x}.
\end{align*}
By using the previous estimates of $f'(|\psi_{n-1}|^2)$, $S_{n-1}$, $V_{n-1}$ and property (1), it gives
\[
\frac{d}{dt}\int_\T|\d_x\psi_n|^2dx\leq \frac{C}{\tau}.
\]
Last, by using the initial condition \eqref{eq:ini_1} we obtain
\[
\int_\T|\d_x\psi_n|^2(t)dx\leq \int_\T|\d_x\psi_0|^2+\frac{CT_*}{\tau}\leq (M_0^\frac12-\delta M_0^{-\frac12})^2+\frac{CT_*}{\tau},
\]
which allows us to choose $T_*$ small such that \eqref{eq:bd_pc} holds. Thus we prove property (2). 

Last, we prove the remaining part of property (1), namely $\|S_n\|_{L^\infty_tH^2_x}\leq C$. By $\d_xS_n=v_n$, we only need to control $\|S_n\|_{L^\infty_tL^2_x}$ and $\|\d_x v_n\|_{L^\infty_tL^2_x}$.
From \eqref{eq:defS_n} we obtain
\[
|S_n(t,x)|\leq S_*+ \int_\T|v_0|dy+\int_0^t\int_T(|\mu_n|+\tau^{-1}|S_{n-1}|)dx_*ds+\int_\T|v_n|dy.
\]
Recalling that $v_n=\IM(\frac{\d_x\psi_n}{\psi_n})$, by the bound of $\psi_n$ and (2) it implies
\[
\int_\T|v_n|dy\leq 2\delta^{-\frac12}\int_\T |\d_x\psi_n|dy\leq \frac{C}{4}.
\]
Same bound holds for $\int_\T|v_0|dy$. By the assumption of induction we have $\|S_{n-1}\|_{L^\infty_tL^2_x}\leq C$, therefore
\[
\int_0^t\int_T\tau^{-1}|S_{n-1}|dx_*ds\leq \frac{C}{4}
\]
if we choose $T_*\leq \frac{\tau}{4}$. Last, by \eqref{eq:mu_n} we have
\[
\int_0^t\int_\T|\mu_n|dx_*ds\leq 2\delta^{-\frac12}\int_0^t\int_\T|\d_x^2\psi_n|dx_*ds+\int_0^t\int_\T|f'(|\psi_{n-1}|^2)|dx_*ds\leq \frac{C}{4},
\]
where in the last inequality we use $f\in C^2([0,\infty))$ and 
\[
|f'(|\psi_{n-1}|^2)|\leq C(\|f\|_{C^2},\|\psi_{n-1}\|_{L^\infty_{t,x}})\leq C(\|f\|_{C^2},\|\psi_{n-1}\|_{L^\infty_{t}H^2_x}).
\]
Thus we prove $\|S_n\|_{L^\infty_{t,x}}\leq C$. On the other hand, again by the definition of $v_n$, we have
\[
\|\d_xv_n\|_{L_t^\infty L^2_x}\leq 2\delta^{-\frac12}\|\d_x^2\psi_n\|_{L_t^\infty L^2_x}+2\delta^{-\frac12}\|(\d_x\psi_n)^2\|_{L_t^\infty L^2_x},
\]
where by embedding inequality,
\[
\|(\d_x\psi_n)^2\|_{L_t^\infty L^2_x}\leq \|\psi_n\|_{L_t^\infty H^2_x}^2\leq 4\|\psi_0\|_{H^2_x}^2\leq C.
\]
Thus we finish the proof of property (1).

Now we need to show $\{(\psi_n,S_n)\}$ is a contractive sequence, which is proved by a analogous argument as property (1) by considering the difference of \eqref{eq:duhamel} and its $L^\infty_tH^2_x$ norm. Since the estimates are standard and almost repetitive, we will not give the details.

\begin{lem}\label{lem:contr}
There exists $0<q<1$, such that the sequence $\{(\psi_n,S_n)\}$ satisfies
\[
\|\psi_n-\psi_{n-1}\|_{L^\infty_tH^2_x}+\|S_{n}-S_{n-1}\|_{L^\infty_tH^2_x}\leq Cq^{n-1}.
\]
\end{lem}

Lemma \ref{lem:contr} imply $\{(\psi_n,S_n)\}$ converges strongly to $(\psi,S)$ in $L^\infty([0,T_*),H^2(\T))$. Moreover by the convergence and the uniform lower bound of $|\psi_n|$, we have
\[
v_n=\IM\left(\frac{\d_x\psi_n}{\psi_n}\right)\to v=\IM\left(\frac{\d_x\psi}{\psi}\right)
\]
\[
-\d_x^2 V_n=|\psi_n|^2-\mathcal{C}(x)\to-\d_x^2 V=|\psi|^2-\mathcal{C}(x)
\]
and
\[
\mu_n=-\frac12\RE\left(\frac{\d_x^2\psi_n}{\psi_n}\right)+f'(|\psi_{n-1}|^2)+V_n\to -\frac12\RE\left(\frac{\d_x^2\psi}{\psi}\right)+f'(|\psi|^2)+V=\mu.
\] 
Also we can rewrite \eqref{eq:defS_n} as
\begin{align*}
S_n(t,x)=&e^{-\frac{t}{\tau}}\left(S_*+\int_\T\int_0^{x_*}v_0(y)dydx_*\right)\\
&-\int_0^t \int_\T e^{\frac{s-t}{\tau}}[\mu_n+\tau^{-1}(S_{n-1}-S_n)](s,x_*)dx_*ds+\int_\T\int_{x_*}^x v(t,y)dydx_*,
\end{align*}
then by passing to the limit we see $S$ satisfies \eqref{eq:defS}. Thus we prove $(\psi,S)$ is a solution to the Cauchy problem \eqref{eq:cauchy_NLS} on $[0,T_*)\times\T$, and the continuity in time follows directly from the Duhamel formula.

The proof of uniqueness follows a standard argument for Schr\"odinger equation by considering the difference equation, similar to Lemma \ref{lem:contr} of contraction.
\end{proof}

Let $T^*>0$ be the maximal time of existence of the solution $\psi\in L^\infty([0,T^*);H^2(\T))$ to the Cauchy problem \eqref{eq:cauchy_NLS}, such that $|\psi|$ is strictly positive. Then, it is possible to write $\psi=\sqrt{\rho}e^{iS}$, for some real-valued functions $\rho, S$. Consequently, the hydrodynamic functions associated to $\psi$ are given by $\rho=|\psi|^2$ and $v=\d_xS$, with $J=\IM(\bar\psi\d_x\psi)=\rho v$, and moreover
\[
|\d_x\psi|^2=(\d_x\sqrt\rho)^2+\rho v^2.
\]
Recalling the definitions \eqref{eq:chem} and \eqref{eq:sigma} for the chemical potential $\mu$ and the quantity $\sigma$, and the above definitions of the hydrodynamic variables associated to $\psi$, we may check that it is possible to identify
\begin{equation}\label{eq:mu_psi}
\mu=-\frac12\RE\left(\frac{\d_x^2\psi}{\psi}\right)+f'(|\psi|^2)+V,\quad\sigma=-\frac12\IM\left(\frac{\d_x^2\psi}{\psi}\right).
\end{equation}
This in turn implies that the higher order functional may be equivalently written in terms of $\psi$ as follows
\begin{equation}\label{eq:high_psi}
I(t)=\frac12\int_{\mathbb T}
\left|-\frac12\d_x^2\psi+f'(|\psi|^2)\psi+V\psi\right|^2\,dx.
\end{equation}
Therefore bounds on $I(t)$ imply $H^2-$estimates for $\psi$.

To prove $\psi$ exists globally, namely $T^*=\infty$, we need the following propositions providing the a priori estimates of $\psi$, and for later use we will write these estimates in term of the hydrodynamic quantities. Despite the fact that it is not necessary for the purpose of our analysis, in the Appendix we provide a formal calculation that derives identity \eqref{eq:rl_I} below in a purely hydrodynamic way.

\begin{prop}\label{prop:3.2}
The $\psi\in L^\infty([0,T^*);H^2(\T))$ be a solution to the Cauchy problem \eqref{eq:cauchy_NLS}, such that $\inf_{t,x}|\psi|>0$. Then, for any $t\in[0,T^*)$, $\psi$ satisfies the following properties:
\begin{itemize}
\item[(1)] the total mass is conserved
\begin{equation}\label{eq:cons_mass}
M(t)=\int_\T\rho(t,x) dx\equiv M_0,
\end{equation}
and the total energy, given by \eqref{eq:en}, dissipates 
\begin{equation}\label{eq:en_disp}
E(t)+\frac{1}{\tau}\int_0^t\int_\T \rho v^2 dxds=E(0)=E_0;
\end{equation}
\item[(2)] if the initial data satisfies
\begin{equation}\label{eq:encr}
E_0\leq \frac12(M_0^\frac12-\delta M_0^{-\frac12})^2,
\end{equation}
for some $\delta>0$, where $M_0=\int_\T\rho_0 dx$ is the initial total mass, then 
\[
\inf_{t,x}\rho\geq \delta\quad \textrm{on }[0,T^*)\times\T;
\]

\item[(3)] The higher order functional $I(t)$, defined by \eqref{eq:higher}, is differentiable for almost every $t\in [0,T^*)$ and satisfies
\begin{equation}\label{eq:rl_I}
\frac{d}{dt} I(t)+\frac{1}{\tau}\int_\T\rho \sigma^2 dx=\int_\T \mu\d_tp(\rho)dx+\int_\T\rho\mu \d_tVdx-\frac{1}{\tau}\int_\T\rho v^2\mu dx,
\end{equation}
where $p(\rho)$ is the pressure given by \eqref{eq:pres}.
\end{itemize}
\end{prop}
%\begin{rem}
%We recall that the total energy $E(t)$ and the higher order functional $I(t)$ may be expressed in terms of $\psi$, see \eqref{eq:iniE_w} and \eqref{eq:high_psi}. Consequently, in what follows we are going to prove identities \eqref{eq:en_disp} and \eqref{eq:rl_I} by multiplying equation \eqref{eq:cauchy_NLS} by suitable multipliers depending on $\psi$ and its derivatives. This provides a rigorous proof in the framework of $H^2$ solutions to \eqref{eq:cauchy_NLS}, that correspond to GCP solutions to \eqref{eq:QHD}. On the other hand, it is also instructive to provide a more direct, yet formal, proof by using the definitions \eqref{eq:en}, \eqref{eq:higher}, \eqref{eq:chem} and \eqref{eq:sigma}, and the dynamical system \eqref{eq:QHD}. This will be performed in the Appendix.
%\end{rem}
\begin{proof}
By equation \eqref{eq:cauchy_NLS}, it is straightforward to check the density $\rho=|\psi|^2$ satisfies the continuity equation
\[
\d_t\rho+\d_xJ=0,
\]
which directly implies the conservation of the total mass. To prove the energy relation \eqref{eq:en_disp}, let us multiply equation \eqref{eq:cauchy_NLS} by $-\d_x^2\bar\psi$, take the imaginary part and integrate by parts, 
\[
\frac12\frac{d}{dt}\int_\T|\d_x\psi|^2dx=-\int_\T[f'(\rho)+\tau^{-1}S+V]\IM(\psi\d_x^2\bar\psi)dx,
\]
where
\[
\IM(\psi\d_x^2\bar\psi)=\d_x\IM(\psi\d_x\bar\psi)=-\d_x(\rho v)=\d_t\rho.
\]
Thus we have
\begin{align*}
\frac{d}{dt}\int_\T\frac12|\d_x\psi|^2+f(\rho)dx=&\tau^{-1}\int_\T S\d_x(\rho v) dx-\int_\T V\d_t\rho dx\\
=&-\tau^{-1}\int_\T\rho v^2 dx+\int_\T V\d_t(\d_x^2 V-\mathcal{C}(x))dx\\
=&-\tau^{-1}\int_\T\rho v^2 dx-\frac{d}{dt}\int_\T\frac12(\d_xV)^2dx.
\end{align*}

Property (2) directly follows the Poincar\'e inequality \eqref{eq:poincare} and property (1),
\[
\|\rho-M_0\|_{L^\infty_{t,x}}\leq \frac12\|\d_x\rho\|_{L^\infty_tL^1_x}\leq M_0^\frac12\|\d_x\psi\|_{L^\infty_tL^2_x}\leq (2M_0E_0)^\frac12.
\]
Therefore if \eqref{eq:encr} holds, it follows
\[
\|\rho-M_0\|_{L^\infty_{t,x}}\leq M_0-\delta,
\]
which implies property (2).

In order to prove property (3), we first show for $\psi\in L^\infty_tH^2_x$ with $\inf_{t,x}|\psi|> 0$, the following identity holds, 
\begin{equation}\label{eq:prop_20_1}
\begin{aligned}
\frac12\frac{d}{dt}\int_\T\rho^{-1}[\RE(\bar\psi\d_x^2\psi)-&2W\rho)^2+\IM(\bar\psi\d_x^2\psi)^2]dx\\
=&-\int_\T\IM[\d_x^2(W\bar\psi+\tau^{-1}S\bar\psi)\d_x^2\psi]dx\\
&+2\int_\T [\rho\d_t(W^2)+W^2\d_t \rho ]dx\\
&-2\int_\T \d_tW\RE(\bar\psi\d_x^2\psi)dx-4\int_\T W\RE(\d_t\bar\psi\d_x^2\psi)dx\\
&-4\int_\T \d_x W\RE(\d_x\bar\psi\d_t\psi)dx-2\int_\T \d_x^2W\RE(\bar\psi\d_t\psi)dx.
\end{aligned}
\end{equation}
where we define the potential $W=f'(\rho)+V$. However, to establish this identity rigorously, we need to introduce the standard sequence of mollifiers $\{\chi_\eps\}_{\eps>0}$ on $\T$ and let $g_\eps=g\ast \chi_\eps$ for functions $g\in L^2(\T)$. We mollify equation \eqref{eq:cauchy_NLS}, multiply it by $\d_x^4\bar\psi_\eps$ and integrate by parts in the imaginary part, which gives
\begin{equation}\label{eq:3.16}
\frac12\frac{d}{dt}\int_\T|\d_x^2\psi_\eps|^2dx=-\int_\T\IM[\d_x^2(W\bar\psi+\tau^{-1}S\bar\psi)_\eps\d_x^2\psi_\eps]dx,
\end{equation}
where the function $S$ is defined by \eqref{eq:defS}. Since $\psi\in L^\infty_tH^2_x$ and $\inf_{t,x}|\psi|> 0$, it is easy to check 
\[
W\bar\psi+\tau^{-1}S\bar\psi\in L^\infty_tH^2_x.
\]
Now we expand the integrand in the left hand side of \eqref{eq:3.16} as
\begin{align*}
|\d_x^2\psi_\eps|^2=|\psi_\eps|^{-2}|\bar\psi_\eps\d_x^2\psi_\eps|^2=\rho_\eps^{-1}[\RE(\bar\psi_\eps\d_x^2\psi_\eps)^2+\IM(\bar\psi_\eps\d_x^2\psi_\eps)^2],
\end{align*}
where  we let $\rho_\eps=|\psi_\eps|^2$, which is different from $\rho\ast\chi_\eps$. To match the chemical potential \eqref{eq:mu_psi}, we write
\begin{align*}
[\RE(\bar\psi_\eps\d_x^2\psi_\eps)]^2=&[\RE(\bar\psi_\eps\d_x^2\psi_\eps)-2W_\eps\rho_\eps]^2\\
&-4W_\eps^2\rho_\eps^2+4W_\eps\rho_\eps\RE(\bar\psi_\eps\d_x^2\psi_\eps).
\end{align*}
thus we have
\begin{align*}
\frac12\frac{d}{dt}\int_\T|\d_x^2\psi_\eps|^2dx=&\frac12\frac{d}{dt}\int_\T \rho_\eps^{-1}[\RE(\bar\psi_\eps\d_x^2\psi_\eps)-2W_\eps\rho_\eps]^2dx\\
&+\frac12\frac{d}{dt}\int_\T \rho_\eps^{-1}\IM(\bar\psi_\eps\d_x^2\psi_\eps)^2dx\\
&+2\frac{d}{dt}\int_\T [W_\eps\RE(\bar\psi_\eps\d_x^2\psi_\eps)-W_\eps^2\rho_\eps]dx.
\end{align*}
Now we compute the time derivative of the last integral,
\begin{align*}
\frac{d}{dt}\int_\T W_\eps\RE(\bar\psi_\eps\d_x^2\psi_\eps)dx=&\int_\T \d_tW_\eps\RE(\bar\psi_\eps\d_x^2\psi_\eps)dx+\int_\T W_\eps\RE(\d_t\bar\psi_\eps\d_x^2\psi_\eps)dx\\
&+\int_\T W_\eps\RE(\bar\psi_\eps\d_x^2\d_t\psi_\eps)dx\\
=&\int_\T \d_tW_\eps\RE(\bar\psi_\eps\d_x^2\psi_\eps)dx+2\int_\T W_\eps\RE(\d_t\bar\psi_\eps\d_x^2\psi_\eps)dx\\
&+2\int_\T \d_xW_\eps\RE(\d_x\bar\psi_\eps\d_t\psi_\eps)dx+\int_\T \d_x^2W_\eps\RE(\bar\psi_\eps\d_t\psi_\eps)dx,
\end{align*}
and
\[
\frac{d}{dt}\int_\T W_\eps^ 2\rho_\eps dx=\int_\T [2\d_t (W_\eps^2)\rho_\eps+W_\eps^2\d_t \rho_\eps] dx.
\]
Since we have $\psi\in L^\infty_tH^2_x$, which implies $\d_t\psi\in L^\infty_tL^2_x$,  the mollified sequences converge strongly in the same topologies respectively, and the integrals are uniformly bounded with respect to $\eps$. Thus we can pass to the limit as $\eps\to 0$, which gives \eqref{eq:prop_20_1}.

In order to reduce \eqref{eq:prop_20_1} to \eqref{eq:rl_I}, it follows \eqref{eq:mu_psi} that
\[
\RE(\bar\psi\d_x^2\psi)=-2\rho\mu+2\rho W,\quad \IM(\bar\psi\d_x^2\psi)=2\rho\sigma,
\]
therefore 
\[
\frac12\frac{d}{dt}\int_\T\rho^{-1}[(\RE(\bar\psi\d_x^2\psi)-2\rho W)^2+\IM(\bar\psi\d_x^2\psi)^2]dx=4\frac{d}{dt}I(t).
\]
We further compute that
\begin{align*}
-\IM[\d_x^2(W\bar\psi+\tau^{-1}S\bar\psi)\d_x^2\psi]=&-[\d_x^2W+\tau^{-1}\d_x v]\IM(\bar\psi\d_x^2\psi)\\
&-2[\d_xW+\tau^{-1}v]\IM(\d_x\bar\psi\d_x^2\psi),
\end{align*}
where we use $\d_x S=v$, and 
\[
-2\d_t W\RE(\bar\psi\d_x^2\psi)=4\d_t W\rho\mu-2\rho\d_t (W^2).
\]
By equation \eqref{eq:cauchy_NLS}, we also have
\[
-4W\RE(\d_t\bar\psi\d_x^2\psi)=4W(W+\tau^{-1}S)\IM(\bar\psi\d_x^2\psi),
\]
\[
-4\d_x W\RE(\d_x\bar\psi\d_t\psi)=2\d_xW\IM(\d_x\bar\psi\d_x^2\psi)-4\d_xW(W+\tau^{-1}S)\IM(\bar\psi\d_x\psi),
\]
and
\[
-2\d_x^2W\RE(\bar\psi\d_t\psi)=\d_x^2W\IM(\bar\psi\d_x^2\psi).
\]
Notice that
\[
\IM(\bar\psi\d_x^2\psi)=\d_x\IM(\bar\psi\d_x\psi)=\d_x(\rho v)=-\d_t\rho,
\]
then by summarising the identities above and integrating by parts, we obtain
\begin{align*}
4\frac{d}{dt}I(t)=&4\int_\T\rho\mu\d_tWdx-\tau^{-1}\int_T\d_xv\d_x(\rho v)dx\\
&+4\tau^{-1}\int_\T \rho v^2Wdx-2\tau^{-1}\int_\T v\IM(\d_x\bar\psi\d_x^2\psi)dx.
\end{align*}
Last, by using
\begin{align*}
\IM(\d_x\bar\psi\d_x^2\psi)=&\rho\IM\left(\frac{\d_x\bar\psi}{\bar\psi}\frac{\d_x^2\psi}{\psi}\right)\\
=&\rho\left[\IM\left(\frac{\d_x\bar\psi}{\bar\psi}\right)\RE\left(\frac{\d_x^2\psi}{\psi}\right)+\RE\left(\frac{\d_x\bar\psi}{\bar\psi}\right)\IM\left(\frac{\d_x^2\psi}{\psi}\right)\right]\\
=&\rho[v(2\mu-2W)+\frac12\rho^{-2}\d_x\rho\d_x(\rho v)],
\end{align*}
it follows that
\[
\frac{d}{dt}I(t)=\int_\T\rho\mu\d_tWdx-\tau^{-1}\int_\T\rho v^2\mu dx-\frac{1}{4\tau}\int_\T \frac{[\d_x(\rho v)]^2}{\rho}dx,
\]
which proves \eqref{eq:rl_I} by using \eqref{eq:pres} and
\[
\frac{[\d_x(\rho v)]^2}{\rho}=\frac{(\d_t\rho)^2}{\rho}=4\rho\sigma^2.
\]
\end{proof}

%\begin{rem}
%The energy balance law \eqref{eq:en_disp} and the time derivative \eqref{eq:rl_I} of the function $I(t)$ can also be derived, at least formally, from the hydrodynamic system \eqref{eq:QHD_1d}. The formal computation is presented in Proposition \ref{lem:dte} and \ref{prop:dtI} in Section \ref{sect:rs} for the rescaled hydrodynamic system \eqref{eq:QHD_rs_intro}, which also gives \eqref{eq:en_disp} and \eqref{eq:rl_I} by applying an inverse scaling.
%\end{rem}

The next proposition provides the boundedness of the functionals $I(t)$ up to any finite time $0\leq t<T^*$. Moreover we will see in Section \ref{sect:rs} that these estimates remain uniform in $\tau$ after the scaling \eqref{eq:rs_intro}, and play essential roles when we consider the relaxation-time limit. Indeed, as $\tau\to 0$, the time-dependent constant remains uniformly bounded on the time interval of length $\mathcal{O}(\tau^{-1})$.

\begin{prop}\label{prop:bdI}
Let $\psi\in L^\infty(0,T^*;H^2(\T))$ be as given in Proposition \ref{prop:3.2}. Then for any $0\leq t<T^*$, the functional $I(t)$ satisfies 
\begin{equation}\label{eq:ineqI}
\begin{aligned}
\frac{d}{dt} I(t)+&\frac{1}{4\tau}\int_\T\rho\sigma^2dx+\frac{1}{2\tau}\int_\T \rho v^4 dx \\
\leq & C(M_0,E_0)\tau \int_\T\rho\mu^2 dx+C(M_0,E_0,\delta)(1+I(t)^\frac12)\frac{1}{\tau}\int_\T \rho v^2dx.
\end{aligned}
\end{equation}
As a consequence we have
\begin{equation}\label{eq:bdI}
I(t)+\frac{1}{\tau}\int_0^t\int_\T \rho\sigma^2 dxds+\frac{1}{\tau}\int_0^t\int_\T \rho v^4 dxds\leq C(M_0,E_0,I_0,\delta,\tau t)
\end{equation}
where $I(0)\leq I_0$ and $C(\dots,\tau t)$ may grow at rate $e^{\tau t}$.
\end{prop}

\begin{proof}
Let us recall \eqref{eq:rl_I} in Proposition \ref{prop:3.2}, the time derivative of functional $I(t)$ is given by
\[
\frac{d}{dt} I(t)+\frac{1}{\tau}\int_T\rho \sigma^2 dx=\int_\T \mu\d_tp(\rho)dx+\int_\T\rho\mu\d_tV dx-\frac{1}{\tau}\int_\T\rho v^2\mu dx.
\]
The right hand side can be estimated in the following way. First we have
\[
\int_\T \mu\d_tp(\rho)dx\leq 2\|p'(\rho)\|_{L^\infty_x}\|\sqrt\rho\mu\|_{L^2_x}\|\d_t\sqrt\rho\|_{L^2_x}.
\]
Recalling $p\in C^2((0,\infty))$, by property (1) in Proposition \ref{prop:3.2} and the Gagliardo-Nirenberg inequality
\[
\|\sqrt\rho\|_{L^\infty_x}\leq C(\|\sqrt\rho\|_{L^2_x}^\frac12\|\d_x\sqrt\rho\|_{L^2_x}^\frac12+\|\sqrt\rho\|_{L^2_x}),
\]
it follows that
\[
\|p'(\rho)\|_{L^\infty_x}\leq C(M_0,E_0).
\]
Thus we obtain
\begin{align*}
\int_\T \mu\d_tp(\rho)dx\leq & C(M_0,E_0)\|\sqrt\rho\mu\|_{L^2_x}\|\d_t\sqrt\rho\|_{L^2_x}\\
\leq & C(M_0,E_0)\tau\int_\T\rho\mu^2dx+\frac{1}{4\tau}\int_\T\rho\sigma^2dx\\
\leq & C(M_0,E_0)\tau I(t)+\frac{1}{4\tau}\int_\T\rho\sigma^2dx
\end{align*}
Similarly, we have
\[
\int_\T \rho\mu\d_tVdx\leq 2\|\sqrt\rho\|_{L^\infty_x}\|\sqrt\rho\mu\|_{L^2_x}\|\d_tV\|_{L^2_x},
\]
Recall that $\int_\T V dx=0$ and $-\d_x^2\d_tV=\d_t\rho$, then by Poincar\'e inequality it follow that
\[
\|\d_tV\|_{L^2_x}\leq C\|\d_t\rho\|_{L^2_x}=2C\|\sqrt\rho\|_{L^\infty_x}\|\d_t\sqrt\rho\|_{L^2_x}\leq C(M_0,E_0)\|\sqrt\rho\sigma\|_{L^2_x},
\]
which gives
\begin{align*}
\int_\T \rho\mu\d_tVdx\leq & C(M_0,E_0)\|\sqrt\rho\mu\|_{L^2_x}\|\sqrt\rho\sigma\|_{L^2_x}\\
\leq & C(M_0,E_0)\tau I(t)+\frac{1}{4\tau}\int_\T\rho\sigma^2dx.
\end{align*}
For the last integral in the right hand side of \eqref{eq:rl_I}, by recalling
\[
\rho\mu=-\frac14\d_x^2\rho+\frac12(\d_x\sqrt\rho)^2+\frac12\rho v^2+f'(\rho)\rho+\rho V,
\]
we write 
\begin{align*}
-\tau^{-1}\int_\T\rho v^2\mu dxds=&\frac{1}{4\tau}\int_\T v^2\d_x^2\rho dx-\frac{1}{2\tau}\int_\T v^2(\d_x\sqrt\rho)^2dx\\
&-\frac{1}{2\tau}\int_\T \rho v^4 dx-\tau^{-1}\int_\T (f'(\rho)+V)\rho v^2dx.
\end{align*}
By integrating by parts, we have
\begin{align*}
\frac{1}{4\tau}\int_\T v^2\d_x^2\rho dx=&-\frac{1}{2\tau}\int_\T v \d_x v\d_x\rho dx\\
=&-\frac{1}{2\tau}\int_\T\frac{\d_x(\rho v)}{\rho}v\d_x\rho dx+\frac{2}{\tau}\int_\T v^2(\d_x\sqrt\rho)^2dx\\
=&\frac{2}{\tau}\int_\T(\sqrt\rho\sigma)( v\d_x\sqrt\rho) dx+\frac{2}{\tau}\int_\T v^2(\d_x\sqrt\rho)^2dx\\
\leq & \frac{1}{4\tau} \int_\T\rho\sigma^2dx+\frac{6}{\tau}\int_\T v^2(\d_x\sqrt\rho)^2dx.
\end{align*}
The integral $\frac{1}{\tau}\int_\T v^2(\d_x\sqrt\rho)^2dx$ can be estimated as
\begin{align*}
\frac{1}{\tau}\int_\T v^2(\d_x\sqrt\rho)^2dx\leq \frac{\|\d_x\sqrt\rho\|_{L^\infty_x}^2}{\delta\tau}\int_\T \rho v^2 dx,
\end{align*}
By Gagliardo-Nirenberg inequality (since $\d_x\sqrt\rho$ has mean $0$) and \eqref{eq:en_disp},
\[
\|\d_x\sqrt\rho\|_{L^\infty_x}^2\leq C \|\d_x\sqrt\rho\|_{L^2_x}\|\d_x^2\sqrt\rho\|_{L^2_x}\leq C(E_0)\|\d_x^2\sqrt\rho\|_{L^2_x}.
\]
Again by using the chemical potential and the mass-energy bounds, we have
\begin{align*}
\|\d_x^2\sqrt\rho\|_{L^2_x}\leq &2\|\sqrt\rho_\tau\mu_\tau\|_{L^2_x}+\delta^{-\frac32}\|\rho v\|_{L^4_x}^2+2\|f'(\rho)\sqrt\rho\|_{L^2_x}\\
\leq & 4 I(t)^\frac12+\delta^{-\frac32}\|\rho v\|_{L^4_x}^2+C(M_0,E_0)
\end{align*}
To control $\|\rho v\|_{L^4_x}$, we use the Gagliardo-Nirenberg inequality,
\[
\int_\T (\rho v)^4dx\leq C\|\rho v\|_{L^2_x}^3\|\d_x(\rho v)\|_{L^2_x}+C\|\rho v\|_{L^2_x}^4,
\]
which imply
\begin{align*}
\delta^{-\frac32}\|\rho v\|_{L^4_x}^2\leq& \frac{C(M_0,E_0)}{\delta^\frac32}(\|\sqrt\rho v\|_{L^2_x}^2+\|\sqrt\rho v\|_{L^2_x}^3)+\delta^{-\frac32}\|\sqrt\rho \sigma\|_{L^2_x}\\
\leq & \frac{C(M_0,E_0)}{\delta^\frac32}+\delta^{-\frac32}I(t)^\frac12.
\end{align*}
Thus we obtain
\begin{align*}
\frac{1}{\tau}\int_\T v^2(\d_x\sqrt\rho)^2dx\leq C(M_0,E_0,\delta)(1+I(t)^\frac12)\tau^{-1}\int_\T\rho v^2dx.
\end{align*}
Last, we have
\begin{align*}
\frac{1}{\tau}\int_\T (f'(\rho)+V)\rho v^2dx\leq&  \|f'(\rho)+V\|_{L^\infty_x}\frac{1}{\tau}\int_\T \rho v^2dx\\
\leq & C(M_0,E_0)\frac{1}{\tau}\int_\T \rho v^2dx.
\end{align*}
Summarizing the inequalities above, we obtain
\begin{align*}
\frac{d}{dt} I(t)+\frac{1}{4\tau}\int_\T&\rho\sigma^2dx+\frac{1}{2\tau}\int_\T \rho v^4 dx \\
\leq & C(M_0,E_0)\tau I(t)+C(M_0,E_0,\delta)(1+I(t)^\frac12)\frac{1}{\tau}\int_\T \rho v^2dx.
\end{align*}
Notice that by \eqref{eq:en_disp},
\[
\frac{1}{\tau}\int_0^t\int_\T \rho v^2dxds\leq E_0,
\]
thus by Gronwall inequality it follows that for any $0\leq t<T^*$,
\begin{equation}\label{eq:bdI_rs2}
I(t)+\frac{1}{\tau}\int_0^t\int_\T\rho\sigma^2dxds+\frac{1}{\tau}\int_0^t\int_\T \rho v^4 dxds\leq C(M_0,E_0,I_0,\delta,\tau t),
\end{equation}
where $C(M_0,E_0,I_0,\tau t)$ grows at most exponentially in $\tau t$.
\end{proof}

As a consequence of Proposition \ref{prop:3.2} and Proposition \ref{prop:bdI}, we can extend the local existence result of the Cauchy problem \eqref{eq:cauchy_NLS} to a global one, which proves the global well-posedness of the Schr\"odinger-Langevin equation as Theorem \ref{thm:NLS}.

\begin{proof}[Proof of Theorem \ref{thm:NLS}]
The positivity of $|\psi|$ is already guaranteed by property (2) of Proposition \ref{prop:3.2}, then we only need to show the $\|\psi\|_{L^\infty_tH^2_x}$ norm remains finite upto any finite time $0<T<\infty$. Noticing that by using the chemical potential \eqref{eq:mu_psi}, we have
\[
\int_\T|\d_x^2\psi|^2dx=4\int_\T\rho[(\mu-f'(\rho)-V)^2+\sigma^2]dx\leq C(I(t)+E(t)+M_0),
\]
then the finiteness of $\|\psi\|_{L^\infty_tH^2_x}$ follows from the bounds \eqref{eq:cons_mass}, \eqref{eq:en_disp} and \eqref{eq:bdI} upto any $0<T<\infty$.
\end{proof}

Now we are at the point to prove Theorem \ref{thm:glob2}.

\begin{proof}[Proof of Theorem \ref{thm:glob2}]
For initial data $(\rho_0,J_0)$ given as in Theorem \ref{thm:glob2}, by using Proposition \ref{prop:lift2} we construct a wave function $\psi_0\in H^2(\T)$ associated to $(\rho_0,J_0)$. Moreover we have $\inf_x|\psi_0|=\inf_x\sqrt\rho_0\geq \delta^\frac12$,
\[
\|\d_x\psi_0\|_{L^2_x}dx \leq \sqrt{2E_0}\leq M_0^\frac12-\delta M_0^{-\frac12},
\]
and $\|\psi_0\|_{H^2_x}\leq C(M_0,E_0,I_0)$. Thus by setting $\psi_0$ as initial data and applying Theorem \ref{thm:NLS}, we solve the Cauchy problem for the NLS equation \eqref{eq:cauchy_NLS} to obtain a global solution $\psi\in L^\infty(0,T;H^2(\T))$ for any $0<T<\infty$. 

Now we define the hydrodynamic variable associated to $\psi$ as 
\[
\rho=|\psi|^2,\quad J=\rho v=\IM(\bar\psi\d_x\psi),
\]
and prove $(\rho, J)$ is a weak solution to \eqref{eq:QHD} in the sense of Definition \ref{def:FEWS}. By using \eqref{eq:cauchy_NLS}, direct computation shows
\[
\d_t\rho=2\RE(\bar\psi\d_t\psi)=-\IM(\bar\psi\d_x^2\psi)=-\d_x\IM(\bar\psi\d_x\psi)=-\d_xJ,
\]
namely the continuity equation holds. To prove the momentum equation, we again need to use the standard mollifiers $\{\chi_\eps\}_{\eps>0}$ and define 
\[
\psi_\eps=\psi\ast\chi_\eps,\quad J_\eps=\IM(\psi_\eps\d_x\psi_\eps).
\]
Again we define the potential $W=f'(\rho)+V$. Then $J_\eps$ satisfies the equation
\begin{align*}
\d_tJ_\eps=&\IM(\d_t\bar\psi_\eps\d_x\psi_\eps)+\IM(\bar\psi_\eps\d_x\d_t\psi_\eps)\\
=& -\frac12\RE(\d_x\bar\psi_\eps\d_x^2\psi_\eps)+\RE[(W\bar\psi+\tau^{-1}S\bar\psi)_\eps\d_x\psi_\eps]\\
&+\frac12\RE(\bar\psi_\eps\d_x^3\psi_\eps)-\RE[\bar\psi_\eps\d_x(W\psi+\tau^{-1}S\psi)_\eps].
\end{align*}
Take $\zeta\in\mathcal C^\infty_0([0, T)\times\T)$ to be an arbitrary test function, then by integrating by parts we have
\begin{align*}
\int_0^T\int_\T J_\eps \d_t\zeta dxdt=&-\int_\T J_\eps(0)\zeta(0)dx-\frac12\int_0^T\int_\T|\d_x\psi_\eps|^2\d_x\zeta dxdt\\
&+\frac12\int_0^T\int_\T\RE(\bar\psi_\eps\d_x^2\psi_\eps)\d_x\zeta dxdt\\
&-\int_0^T\int_\T \zeta\RE[(W\bar\psi+\tau^{-1}S\bar\psi)_\eps\d_x\psi_\eps] dxdt\\
&+\int_0^T\int_\T \zeta\RE[\bar\psi_\eps\d_x(W\psi+\tau^{-1}S\psi)_\eps]dxdt.
\end{align*}
By letting $\eps\to 0$ and using 
\[
J_\eps=\IM(\psi_\eps\d_x\psi_\eps)\to \IM(\psi\d_x\psi)=J,
\]
it follows that
\begin{align*}
\int_0^T\int_\T J \d_t\zeta dxdt=&-\int_\T J_0\zeta(0)dx-\frac12\int_0^T\int_\T|\d_x\psi|^2\d_x\zeta dxdt\\
&+\frac12\int_0^T\int_\T\RE(\bar\psi\d_x^2\psi)\d_x\zeta dxdt\\
&-\int_0^T\int_\T \zeta\RE[(W\bar\psi+\tau^{-1}S\bar\psi)\d_x\psi] dxdt\\
&+\int_0^T\int_\T \zeta\RE[\bar\psi\d_x(W\psi+\tau^{-1}S\psi)]dxdt.
\end{align*}
We further compute
\[
\RE(\bar\psi\d_x^2\psi)=\frac12\d_x^2\rho-|\d_x\psi|^2,
\]
and
\begin{align*}
-\RE[(W\bar\psi+\tau^{-1}S\bar\psi)\d_x\psi]+\RE[&\bar\psi\d_x(W\psi+\tau^{-1}S\psi)]\\
=&\rho\d_x [f'(\rho)+V]+\tau^{-1}\rho\d_xS\\
=&\d_xp(\rho)+\rho\d_xV+\tau^{-1}J,
\end{align*}
where in the last identity we use 
\[
p(\rho)=f'(\rho)\rho-f(\rho),\; \d_xS=v.
\]
Thus we obtain
\begin{align*}
\int_0^T\int_\T J \d_t\zeta dxdt=&-\int_\T J_0\zeta(0)dx-\int_0^T\int_\T|\d_x\psi|^2\d_x\zeta dxdt\\
&+\frac14\int_0^T\int_\T\d_x^2\rho\d_x\zeta dxdt-\int_0^T\int_\T \zeta (\d_xp(\rho)+\rho\d_xV+\tau^{-1}J)dxdt,
\end{align*}
then by using the polar factorization Lemma \ref{lem:polar},
\[
|\d_x\psi|^2=(\d_x\sqrt\rho)^2+\Lambda^2,
\]
we prove the momentum equation. Moreover, by Proposition \ref{prop:3.2} and Proposition \ref{prop:bdI}, we see $(\rho,J)$ is a pair of GCP solution as in Definition \ref{def:GCPsln} and satisfies the properties in Theorem \ref{thm:glob2}.

Last, to prove the uniqueness, let us assume $(\rho_1, J_1)$ to be an arbitrary GCP solution to \eqref{eq:QHD} satisfying the property of Theorem \ref{thm:glob2}. Then we can define a wave function
\[
\psi_1=\sqrt{\rho}e^{i S_1},
\]
where $S_1$ is the phase function of $(\rho_1,J_1)$ given by Definition \ref{def:phase}. By equation \eqref{eq:gradS}, it is straightforward to compute that $\psi_1$ is an wave function associated to $(\rho_1,J_1)$ in the sense
\[
\rho_1=|\psi_1|^2,\quad J_1=\IM(\bar\psi_1\d_x\psi_1).
\]
Moreover, we claim $\psi_1$ is a solution of the NLS equation \eqref{eq:cauchy_NLS}. By the continuity equation of $\rho$ and \eqref{eq:gradS}, we have
\begin{align*}
\d_t\psi_1=&(\d_t\sqrt\rho_1+i\sqrt\rho_1\d_tS_1)e^{iS_1}\\
=&-\left[\frac{\d_xJ_1}{2\sqrt\rho_1}+i\sqrt\rho_1(\mu_1+\tau^{-1}S_1)\right]e^{iS_1}.
\end{align*}
Recall that
\[
\d_xJ_1=\IM(\bar\psi_1\d_x^2\psi_1),\quad \mu_1=-\frac12\RE\left(\frac{\d_x^2\psi_1}{\psi_1}\right)+W_1,
\]
where $W_1=f'(\rho_1)+V_1$ and $-\d_x^2 V_1=\rho_1-\mathcal{C}(x)$, then it follows that
\begin{align*}
\d_t\psi_1=&\left[-\frac{\IM(\bar\psi_1\d_x^2\psi_1)}{2\rho_1}+\frac{i}{2}\RE\left(\frac{\d_x^2\psi_1}{\psi_1}\right)-i(W_1+\tau^{-1}S_1)\right]\sqrt\rho_1 e^{iS_1}\\
=&\left[-\frac12\IM\left(\frac{\d_x^2\psi_1}{\psi_1}\right)+\frac{i}{2}\RE\left(\frac{\d_x^2\psi_1}{\psi_1}\right)-i(W_1+\tau^{-1}S_1)\right]\psi_1\\
=&\frac{i}{2}\d_x^2\psi_1-i(W_1+\tau^{-1}S_1)\psi_1.
\end{align*}
Now we consider the initial data $\psi_{1,0}(x)=\psi_1(0,x)$ and $\psi_0$, which are associated to same hydrodynamic data $(\sqrt\rho_0,\Lambda_0=J_0/\sqrt\rho_0)$ in the sense that
\[
|\psi_{1,0}|=\sqrt\rho_0=|\psi_0|,\quad \IM(\bar\phi_{0,1}\d_x\psi_{0,1})=\Lambda_0=\IM(\bar\phi_{0}\d_x\psi_{0}).
\]
Then the following lemma from \cite{AMZ1} tell us $\psi_{1,0}$ and $\psi_{0}$ are only differed by a constant phase shift.

\begin{lem}\label{lemma:rot}
Let $\psi_1\in H^1([a,b])$ and $\psi_2\in H^1([a,b])$ be two wave functions associated to the same hydrodynamic data $(\sqrt{\rho},\Lambda)\in H^1([a,b])\times L^2([a,b])$, namely
\[
\sqrt{\rho}=|\psi_1|=|\psi_2|,\quad \Lambda=\IM(\bar\phi_1\d_x\psi_1)=\IM(\bar\phi_2\d_x\psi_2),
\] 
where $\phi_j$ is a polar factor of $\psi_j$, $j=1,2$. Furthermore we assume $\rho>0$ on interval $[a,b]$. Then there exists a constant $\theta\in[0,2\pi)$ such that
\[
\psi_2=e^{i\theta}\,\psi_1.
\]
\end{lem}

Therefore by choosing $S_1(0,0)$ in Definition \ref{def:phase} suitably, we can make $\psi_{1,0}=\psi_{0}$. Then by the uniqueness the Cauchy problem \eqref{eq:cauchy_NLS}, we have $\psi_1=\psi$, which implies
\[
(\rho_1,J_1)=(|\psi_1|^2,\IM(\bar\psi_1\d_x\psi_1))=(|\psi|^2,\IM(\bar\psi\d_x\psi))=(\rho,J).
\]
\end{proof}

\section{Rescaling of QHD system with damping and uniform estimates in $\tau$}\label{sect:rs}

From this section we begin to consider the rescaling of the QHD system \eqref{eq:QHD} by the relaxation-time $\tau$ and the limiting system as $\tau\to 0$. By using the scaling
\begin{equation}\label{eq:rs}
t'=\tau t,\quad (\rho_\tau,v_\tau)(t',x)=\left(\rho,\frac{1}{\tau}v\right)\left(\frac{t'}{\tau},x\right),
\end{equation}
system \eqref{eq:QHD} can be rewritten as
\begin{equation}\label{eq:QHD_rs}
\left\{\begin{aligned}
&\d_{t'}\rho_\tau+\d_x (\rho_\tau v_\tau)=0\\
&\tau^2\d_t (\rho_\tau v_\tau)+\tau^2\d_x(\rho_\tau v_\tau^2)+\d_x p(\rho_\tau)+\rho_\tau\d_x V_\tau=\frac{1}{2}\rho_\tau\d_x\left(\frac{\d_x^2\sqrt\rho_\tau}{\sqrt\rho_\tau}\right)-\rho_\tau v_\tau\\
&-\d_x^2V_\tau=\rho_\tau-\mathcal{C}(x).
\end{aligned}\right.
\end{equation}
For simplicity of the notation, in the later contents of this paper we still use $t$ to denote the rescaled time, but We emphasize that from now on the time $t$ is the rescaled time.

By using rescaling \eqref{eq:rs}, the total mass $M(t)$, the total energy $E(t)$ and the higher order functional $I(t)$ can be rewritten into rescaled forms respectively, which provide a series of bounds uniform in $\tau$.

First, conservation of mass \eqref{eq:mass} and energy relation \eqref{eq:en} are transformed into
\begin{equation}\label{eq:M_rs}
M_\tau(t)=\int_\T \rho_\tau(t)dx\equiv M_0,
\end{equation}
and
\begin{equation}\label{eq:en_rs}
E_\tau(t)+\int_0^T\int_\T \rho_\tau v_\tau^2dxds=E_0,
\end{equation}
for $E_\tau(t)=\int_\T e_\tau(t)dx$, where the rescaled energy density is given by
\begin{equation}\label{eq:endens_rs}
e_\tau=\frac{\tau^2}{2}\rho_\tau v_\tau^2+\frac12(\d_x\sqrt{\rho_\tau})^2+f(\rho_\tau)+\frac12(\d_xV_\tau)^2.
\end{equation}

For the higher or order functionals, we define the rescaled chemical potential as
\begin{equation}\label{eq:chem_rs}
\mu_\tau=\frac{1}{\tau}\left(-\frac{\d_x^2\sqrt\rho_\tau}{2\sqrt\rho_\tau}+\frac{\tau^2}{2}v_\tau^2+f'(\rho_\tau)+V_\tau\right),\quad \sigma_\tau=\d_t\log\sqrt{\rho}.
\end{equation}
Moreover, the rescaled phase function $S_\tau$ is introduced following the idea of Definition \ref{def:phase}, which should satisfies the space-time gradient equation
\begin{equation}\label{eq:gradS_rs}
\begin{cases}
\d_xS_\tau =v_\tau,\quad\tau^2\d_t S_\tau=-\tau\mu_\tau-S_\tau \\
S_\tau(0,0)=\tau^{-1}S_*.
\end{cases}
\end{equation}
Last, the higher order functional are rescaled as
\begin{equation}\label{eq:higher_rs}
I_\tau(t)=\int_T\frac{\tau^2}{2}\rho_\tau(\mu_\tau^2+\sigma_\tau^2)dx.
\end{equation}

The uniform estimate of $I_\tau(t)$ can be obtained by rescaling \eqref{eq:rl_I} and Proposition \ref{prop:bdI}.

\begin{prop}\label{prop:bdI_rs}
Let $(\rho_\tau, v_\tau)$ be a finite energy GCP solution to \eqref{eq:QHD_rs} such that $\rho_\tau>0$. Then for almost every $0\leq t<\infty$, the functional $I_\tau(t)$ is differentiable, 
\begin{equation}\label{eq:dispI_rs}
\begin{aligned}
\frac{d}{dt}I_\tau(t)+\int_\T \rho_\tau\sigma_\tau^2 dx=&\tau\int_\T\mu_\tau\d_tp(\rho_\tau)dx\\
&+\tau\int_\T\rho_\tau\mu_\tau\d_tV_\tau dx-\tau\int_\T\rho_\tau v_\tau^2\mu_\tau dx.
\end{aligned}
\end{equation}
As a consequence, it follows that
\begin{equation}\label{eq:ineqI_rs}
\begin{aligned}
\frac{d}{dt} I_\tau(t)+&\frac14\int_\T\rho_\tau\sigma_\tau^2dx+\frac{\tau^2}{2}\int_\T \rho_\tau v_\tau^4 dx \\
\leq & C(M_0,E_0)\int_\T\tau^2\rho_\tau\mu_\tau^2 dx+C(M_0,E_0,\delta)(1+I_\tau(t)^\frac12)\int_\T \rho_\tau v_\tau^2dx.
\end{aligned}
\end{equation}
Therefore we have the bounds uniform with respect to $\tau$ that
\begin{equation}\label{eq:bdI_rs}
I_\tau(t)+\int_0^t\int_\T \rho_\tau\sigma_\tau^2 dxds+\tau^2\int_0^t\int_\T \rho_\tau v_\tau^4 dxds\leq C(M_0,E_0,I_0,\delta,t),
\end{equation}
where $I_\tau(0)\leq I_0$ and $C(\dots,t)$ may grow in time at rate $e^{t}$.
\end{prop}

\section{Entropy estimates and Dissipation}\label{sect:entrp}

Inspired by the property of the quantum drift-diffusion equation \eqref{eq:qdde}, let us introduce the physical entropy
\begin{equation}\label{eq:entr_1}
H(\rho_\tau)=\int_\T\rho_\tau\log\left(\frac{\rho_\tau}{M_0}\right)dx,
\end{equation}
where $M_0$ is the conserved total mass of $\rho_\tau$. As discussed before, the convexity of the function $g(s)=s\log s$ and the assumption $|\T|=1$ implies $H(\rho_\tau)\geq 0$.

We first compute the time derivative of entropy $H(\rho_\tau)$ for GCP solutions, and by combining it with Proposition \ref{prop:bdI_rs}, a direct consequence is the boundedness of the entropy $H(\rho_{\tau})$. We emphasize again that here $t$ is the rescaled time given by \eqref{eq:rs}.

\begin{prop}\label{prop:entr}
Let $(\rho_\tau,v_\tau)$ be a finite energy GCP solution to the rescaled QHD system \eqref{eq:QHD_rs} such that $\inf_{t,x}\rho_\tau\geq \delta>0$. 
The time derivative of $H(\rho_\tau)$ satisfies the inequality
\begin{equation}\label{eq:ineq_entr_rs}
\begin{aligned}
\frac{d}{dt}&[H(\rho_\tau)+\tau^2\int_\T\log\rho_\tau\d_t\rho_\tau dx]+\frac12\int_\T\rho_\tau(\d_x^2\log\sqrt\rho_\tau)^2dx\\
+&4\int_\T p'(\rho_\tau)(\d_x\sqrt\rho_\tau)^2dx+\int_\T\rho_\tau(\rho_\tau-\mathcal{C}(x))dx\leq4\tau^2\int_\T\rho_\tau\sigma_\tau^2dx+\tau^4\int_\T \rho_\tau v_\tau^4dx.
\end{aligned}
\end{equation}
If we also assume $I_\tau(0)\leq I_0<\infty$, then if follows that for any $0\leq t<\infty$,
\begin{equation}\label{eq:en_bd1}
\begin{aligned}
&\left.H(\rho_{\tau})(s)\right|_{s=0}^{s=t}+\int_0^t\int_\T\rho_\tau(\d_x^2\log\sqrt\rho_\tau)^2dxds\\
&+\int_0^t\int_\T p'(\rho_\tau)(\d_x\sqrt\rho_\tau)^2dxds+\int_\T\rho_\tau(\rho_\tau-\mathcal{C}(x))dx\leq C(E_0)\tau,
\end{aligned}
\end{equation}
where we use the notation
\[
\left.H(\rho_{\tau})(s)\right|_{s=0}^{s=t}=H(\rho_{\tau})(t)-H(\rho_{\tau})(0).
\]
\end{prop}

\begin{proof}
To rigorously prove \eqref{eq:ineq_entr_rs} for GCP solutions, we again need to use the standard mollifiers $\{\chi_\eps\}_{\eps>0}$. Let $\rho_{\tau,\eps}=\rho_\tau\ast\chi_\eps$ and we consider the entropy $H(\rho_{\tau,\eps})$ of the mollified density. By mollifying equation \eqref{eq:QHD_rs} and using \eqref{eq:bohm} to write the quantum term in a logarithm form, the time derivative of $H(\rho_{\tau,\eps})$ is computed as
\begin{align*}
\frac{d}{dt}H(\rho_{\tau,\eps})=&\int_\T\log\rho_{\tau,\eps}\d_t\rho_{\tau,\eps} dx=-\int_\T\log\rho_{\tau,\eps}\d_x(\rho_\tau v_\tau)_\eps dx\\
=&-\frac12\int_\T\log\rho_{\tau,\eps}\d_x^2(\rho_\tau\d_x^2\log\sqrt\rho_\tau)_\eps dx+\int_\T\log\rho_{\tau,\eps}\d_x^2p(\rho_\tau)_\eps dx\\
&+\tau^2\int_\T\log\rho_{\tau,\eps}\d_t\d_x(\rho_\tau v_\tau)_\eps dx+\tau^2\int_\T\log\rho_{\tau,\eps}\d_x^2(\rho_\tau v_\tau^2)_\eps dx\\
&+\int_\T\log\rho_{\tau,\eps}\d_x(\rho_\tau \d_xV_\tau)_\eps dx\\
=&-\int_\T(\rho_\tau\d_x^2\log\sqrt\rho_\tau)_\eps \d_x^2\log\sqrt\rho_{\tau,\eps}dx-\int_\T \d_xp(\rho_\tau)_\eps\d_x\log \rho_{\tau,\eps}dx\\
&-\tau^2\int_\T\log\rho_{\tau,\eps}\d_t^2\rho_{\tau,\eps} dx+\tau^2\int_\T(\rho_\tau v_\tau^2)_\eps\d_x^2\log\rho_{\tau,\eps} dx\\
&-\int_\T\frac{\d_x\rho_{\tau,\eps}}{\rho_{\tau,\eps}}(\rho_\tau \d_xV_\tau)_\eps dx.
\end{align*}
where we can further write 
\[
-\tau^2\int_\T\log\rho_{\tau,\eps}\d_t^2\rho_{\tau,\eps} dx=-\tau^2\frac{d}{dt}\int_\T\log\rho_{\tau,\eps}\d_t\rho_{\tau,\eps} dx+\tau^2\int_\T\d_t\log\rho_{\tau,\eps}\d_t\rho_{\tau,\eps} dx.
\]
Thus the integrands in the right hand side of $\frac{d}{ds}H(\rho_{\tau,\eps})$ are all mollifications of well-defined functions, and the integrals are uniformly bounded with respect to $\eps$. Thus we can pass to the limit $\eps\to 0$ to obtain
\begin{align*}
\frac{d}{dt}H(\rho_\tau)=&-\int_\T\rho_\tau(\d_x^2\log\sqrt\rho_\tau)^2 dx-\int_\T \d_xp(\rho_\tau)\d_x\log \rho_\tau dx\\
&-\tau^2\frac{d}{dt}\int_\T\log\rho_\tau\d_t\rho_\tau dx+\tau^2\int_\T\d_t\log\rho_\tau\d_t\rho_\tau dx\\
&+\tau^2\int_\T\rho_\tau v_\tau^2\d_x^2\log\rho_\tau dx-\int_\T\d_x\rho_\tau\d_xV_\tau dx\\
=&-\int_\T\rho_\tau(\d_x^2\log\sqrt\rho_\tau)^2 dx-4\int_\T p'(\rho_\tau)(\d_x\sqrt\rho_\tau)^2dx\\
&-\tau^2\frac{d}{dt}\int_\T\log\rho_\tau\d_t\rho_\tau dx+4\tau^2\int_\T\rho_\tau\sigma_\tau^2 dx\\
&+\tau^2\int_\T\rho_\tau v_\tau^2\d_x^2\log\rho_\tau dx+\int_\T\rho_\tau\d_x^2V_\tau dx.
\end{align*}
For the last two integrals, by Proposition \ref{prop:bdI_rs}, we have
\begin{align*}
\tau^2\left|\int_\T\rho_\tau v_\tau^2\d_x^2\log\rho_\tau dx\right|\leq& 2\tau^2\|\rho_\tau^\frac14 v_\tau\|_{L^4_x}^2\|\sqrt\rho_\tau\d_x^2\log\sqrt\rho_\tau\|_{L^2_x}\\
\leq &2\tau^4\int_\T \rho_\tau v_\tau^4dx+\frac12\int_\T\rho_\tau(\d_x^2\log\sqrt\rho_\tau)^2dx,
\end{align*}
and
\[
\int_\T\rho_\tau\d_x^2V_\tau dx=-\int_\T\rho_\tau(\rho_\tau-\mathcal{C}(x))dx,
\]
which finishes the proof of \eqref{eq:ineq_entr_rs}. Last, by integrating \eqref{eq:ineq_entr_rs} in time, the integral $\tau^2\int_\T\log\rho_\tau\d_t\rho_\tau dx$ is estimated in the following way,
\begin{equation*}\begin{aligned}
\tau^2\int_\T\log\rho_\tau\d_t\rho_\tau\,dx
=\tau^2\int_\T \rho_\tau v_\tau\d_x\log\rho_\tau\,dx
=2\tau^2\int_\T\sqrt{\rho_\tau}v_\tau\d_x\sqrt{\rho_\tau}\,dx
\leq\tau E_\tau(t).
\end{aligned}\end{equation*}
Thus we prove \eqref{eq:en_bd1}.
\end{proof}

In Proposition \ref{prop:bdI_rs} and their consequential estimates in Proposition \ref{prop:entr}, the upper bounds $C(\cdot,t)$ may grow at most exponentially in $t$. From \eqref{eq:en_bd1} we see that this time-dependent term decays to $0$ as $\tau\to 0$.

\subsection{Dissipation for small $\tau$ and internal energy $f(s)=(s-M_0)^{2n}$}\label{sect:disp}

In \cite{JM} and the references therein, the quantum drift-diffusion equation \eqref{eq:qdde} is considered in the pressureless and non-electric case (the DLSS equation), namely $p(\rho)=0$ and $V=0$. If $\bar\rho$ is a suitable solution to the DLSS equation, the physical entropy $\tilde E_{1}(\bar\rho)$ in fact dissipates and has an exponential decay. This result can be extended to equation \eqref{eq:qdde} with non-zero pressure and electric potential, but suitable assumptions are needed on the pressure function $p(s)$ and the background charge density $\mathcal{C}(x)$. First by formally letting $\tau=0$ in \eqref{eq:ineq_entr_rs}, we see that a straightforward condition is to assume $p'(s)\geq 0$ for $s\geq 0$ to guarantee the non-negativity of the integral $\int_\T p'(\bar\rho)(\d_x\sqrt{\bar\rho})^2dx$. By using the relation
\[
p(s)=s f'(s)-f(s),
\]
it follows that $p'(s)=s f''(s)$, therefore $p'(s)\geq 0$ is equivalent to the convexity of $f(s)$. On the other hand. the dissipation of $\bar\rho(\d_x^2\log\sqrt{\bar\rho})^2$ in \eqref{eq:ineq_entr_rs} (letting $\tau=0$) implies $\bar\rho\to M_0$ as $t\to\infty$, hence we may assume $f(M_0)=0$ to make the internal energy also dissipate. From observation above, in this section we focus on internal energy of the form
\[
f(s)=(s-M_0)^{2n},
\] 
with $n\in\N$. Similarly, to allow the electric potential $V$ to dissipate to a constant state, we should assume the condition that $\mathcal{C}(x)\equiv M_0$.

We will show that for such $f(\rho)$ and $\mathcal{C}(x)$, the dissipation and exponential decay also hold for GCP solution with $\tau$ small, not only for the entropy $\tilde E_{1}(\rho_\tau)$, but also for the total energy $E_\tau(t)$ and the functional $I_\tau(t)$ under the assumption $I_\tau(0)\leq I_0<\infty$, which proves Theorem \ref{thm:disp}.

\begin{prop}\label{prop:decay}
Let $(\rho_\tau, v_\tau)$ be a finite energy GCP solution of \eqref{eq:QHD_rs} with internal energy $f(s)=(s-M_0)^{2n}$, $n\in\N$ and back ground charge density $\mathcal{C}(x)=M_0$. We also define the the functional
\[
F_\tau(t)=H(\rho_\tau)+ E_\tau(t)+c_1 I_\tau(t),
\]
where $ c_1\sim C(E_0,M_0,\delta)^{-1}$ is a fixed small number depending on $M_0$, $E_0$ and $\delta$, then for $0<\tau<\tau_0=\min\left\{c_1,\frac{c_1^\frac12}{4},\left(\frac{2c_1\delta}{8+\delta}\right)^\frac12\right\}$, the functional $F_\tau(t)$ satisfies
\begin{equation}\label{eq:decay}
F_\tau(t)\leq F_\tau(0)\exp\left(C(E_0,M_0,\delta)\int_0^t\int_\T\rho_\tau v_\tau^2dxds- t\right).
\end{equation}
As a consequence, $F_\tau(t)$ decay exponentially for $t$ large.
\end{prop}

\begin{proof}
By using \eqref{eq:en_rs}, \eqref{eq:ineqI_rs} and \eqref{eq:ineq_entr_rs}, the time derivative of $F_\tau(t)$ satisfies
\begin{equation}\label{eq:dtF_1}
\begin{aligned}
\frac{d}{dt}[F_\tau(t)+&\tau^2\int_\T\log\rho_\tau\d_t\rho_\tau dx]+\int_\T\rho_\tau v^2_\tau dx+\left(\frac{c_1}{2}-4\tau^2\right)\int_\T\rho_\tau\sigma_\tau^2dx\\
&+\left(\frac{c_1\tau^2}{2}-\frac{2\tau^4}{\delta}\right)\int_\T \rho_\tau v_\tau^4 dx+\frac12\int_\T\rho_\tau(\d_x^2\log\sqrt\rho_\tau)^2dx\\
&+4\int_\T p'(\rho_\tau)(\d_x\sqrt\rho_\tau)^2dx+\int_\T(\rho_\tau-M_0)^2dx\\
\leq & C(M_0,E_0)c_1 \int_\T\tau^2\rho_\tau\mu_\tau^2 dx+C(M_0,E_0,\delta)c_1(1+I_\tau(t)^\frac12)\int_\T \rho_\tau v_\tau^2dx,
\end{aligned}
\end{equation}
where we use the observation that
\[
\int_\T\rho_\tau(\rho_\tau-M_0)dx=\int_\T(\rho_\tau-M_0)^2dx.
\]
To prove the exponential decay \eqref{eq:decay}, we need to control the right hand side and $F_\tau(t)$ itself by the dissipative terms.

We first notice that
\begin{equation}\label{eq:H^2bd}
\begin{aligned}
\int_\T\rho_\tau(\d_x^2\log\sqrt\rho_\tau)^2 dx=&\int_\T(\d_x^2\sqrt\rho_\tau)^2+\frac{(\d_x\sqrt\rho_\tau)^4}{\rho_\tau}-\frac23\frac{\d_x(\d_x\sqrt\rho_\tau)^3}{\sqrt\rho_\tau} dx\\
=&\int_\T(\d_x^2\sqrt\rho_\tau)^2+\frac{(\d_x\sqrt\rho_\tau)^4}{3\rho_\tau}dx\\
=&\int_\T(\d_x^2\sqrt\rho_\tau)^2+\frac{16}{3}(\d_x\rho_\tau^{\frac14})^4dx.
\end{aligned}
\end{equation}
From \eqref{eq:chem_rs} we can write
\[
-2\tau\sqrt\rho_\tau\mu_\tau=-\d_x^2\sqrt\rho_\tau+\tau^2\sqrt\rho_\tau v^2_\tau+2(f'(\rho_\tau)+V_\tau)\sqrt\rho_\tau,
\]
which implies
\[
2\int_\T \tau^2\rho_\tau\mu_\tau^2 dx\leq \int_\T(\d_x^2\sqrt\rho_\tau)^2dx+\tau^4\int_\T \rho_\tau v^4_\tau dx+4\int_\T (f'(\rho_\tau)^2+V_\tau^2)\rho_\tau dx.
\]
For the last integral, by using $f(\rho_\tau)=(\rho_\tau-M_0)^{2n}$ and the mass-energy bounds \eqref{eq:M_rs}, \eqref{eq:en_rs}, we have
\[
\int_\T f'(\rho_\tau)^2\rho_\tau dx=\int_\T 4n^2\rho_\tau(\rho_\tau-M_0)^{4n-2} dx\leq C(M_0,E_0,n) \int_\T (\rho_\tau-M_0)^{2n} dx,
\]
then by Poincar\'e inequality, it follows that
\[
\int_\T (\rho_\tau-M_0)^{2n} dx\leq \frac{1}{4\pi^2}\int_\T[\d_x(\rho_\tau-M_0)^n]^2dx.
\]
On the other hand, by 
\[
p'(\rho_\tau)=\rho_\tau f''(\rho_\tau)=2n(2n-1)\rho_\tau(\rho_\tau-M_0)^{2n-2},
\]
we have
\begin{align*}
4\int_\T p'(\rho_\tau)(\d_x\sqrt\rho_\tau)^2dx=&\int_\T 2n(2n-1)(\rho_\tau-M_0)^{2n-2}(\d_x\rho_\tau)^2dx\\
=&2(2-n^{-1})\int_\T[\d_x(\rho_\tau-M_0)^n]^2dx.
\end{align*}
Also, by the Poisson equation of $V$ and Poincar\'e inequality, we have
\[
\int_T\rho_\tau V_\tau^2dx\leq \|\rho_\tau\|_{L^\infty_x}\int_T V_\tau^2dx\leq \frac{C(M_0,E_0)}{16\pi^4}\int_T (\rho_\tau-M_0)^2dx.
\]
Thus we can choose $c_1>0$ small such that
\begin{align*}
2C(M_0,E_0)&c_1 \int_\T\tau^2\rho_\tau\mu_\tau^2 dx\\
\leq & C(M_0,E_0)c_1\left[\int_\T(\d_x^2\sqrt\rho_\tau)^2dx+\tau^4\int_\T \rho_\tau v^4_\tau dx+4\int_\T f'(\rho_\tau)^2\rho_\tau dx\right]\\
\leq& \frac14\int_\T(\d_x^2\sqrt\rho_\tau)^2dx+\frac{\tau^4}{4}\int_\T \rho_\tau v^4_\tau dx\\
&+2\int_\T p'(\rho_\tau)(\d_x\sqrt\rho_\tau)^2dx+\frac12\int_\T(\rho_\tau-M_0)^2dx.
\end{align*}
To control the second term in the right hand side of \eqref{eq:dtF_1}, by using
\[
1+I_\tau(t)^\frac12\leq \frac{5}{4}+I_\tau(t),
\]
it is enough to consider the integral
\[
C(M_0,E_0,\delta)c_1(1+I_\tau(t))\int_\T \rho_\tau v_\tau^2dx,
\]
and we update the choice of $c_1>0$ such that $C(M_0,E_0,\delta)c_1<\frac{1}{2}$. Summarizing the inequalities above, we first update \eqref{eq:dtF_1} into
\begin{equation}\label{eq:dtF_2}
\begin{aligned}
\frac{d}{dt}[F_\tau(t)+&\tau^2\int_\T\log\rho_\tau\d_t\rho_\tau dx]+\frac{1}{2}\int_\T\rho_\tau v^2_\tau dx+C(M_0,E_0)c_1 \int_\T\tau^2\rho_\tau\mu_\tau^2 dx\\
&+\left(\frac{c_1}{2}-4\tau^2\right)\int_\T\rho_\tau\sigma_\tau^2dx+\left(\frac{c_1\tau^2}{2}-\frac{2\tau^4}{\delta}-\frac{\tau^4}{4}\right)\int_\T \rho_\tau v_\tau^4 dx\\
&+\frac14\int_\T(\d_x^2\sqrt\rho_\tau)^2dx+\frac{8}{3}\int_\T(\d_x\rho_\tau^\frac14)^4dx
+2\int_\T p'(\rho_\tau)(\d_x\sqrt\rho_\tau)^2dx\\
&+\frac12\int_\T(\rho_\tau-M_0)^2dx\leq  C(M_0,E_0,\delta)c_1 I_\tau(t)\int_\T \rho_\tau v_\tau^2dx,
\end{aligned}
\end{equation}
where in the left hand side we have
\[
C(M_0,E_0)c_1 \int_\T\tau^2\rho_\tau\mu_\tau^2 dx+\left(\frac{c_1}{2}-4\tau^2\right)\int_\T\rho_\tau\sigma_\tau^2dx\geq \frac{c_1}{2} I_\tau(t)
\]
when $\tau\leq \tau_0\leq \frac{c_1^\frac12}{4}$.

Now we need to control $H(\rho_\tau)+E_\tau(t)$ by the dissipative terms in \eqref{eq:dtF_2}. We recall the following technical lemma.

\begin{lem}[Logarithmic Sobolev inequality \cite{DGJ}]\label{lem:logineq}
Let $u\in H^1(\T)$ and $\overline{u^2}=|\T|^{-1}\int_\T u^2 dx$, then
\begin{equation}\label{eq:logineq}
\int_\T u^2\log(u^2/\overline{u^2})dx\leq \frac{1}{2\pi^2}\int_\T (\d_x u)^2dx.
\end{equation}
\end{lem}

By letting $u=\sqrt\rho_\tau$ in Lemma \ref{lem:logineq} and applying again Poincar\'e inequality, we obtain
\[
H(\rho_\tau)\leq \frac{1}{2\pi^2}\int_\T (\d_x \sqrt\rho_\tau)^2dx\leq \frac{1}{8\pi^4}\int_\T (\d_x^2 \sqrt\rho_\tau)^2dx.
\]
For the total energy $E_\tau(\rho)$, the quantum part $\int_\T(\d_x\sqrt\rho_\tau)^2dx$ also follows the inequality above. On the other hand, the internal energy and the electric potential can be bounded by
\[
\int_\T f(\rho_\tau)dx=\int_\T(\rho_\tau-M_0)^{2n}dx \leq \frac{1}{4\pi^2}\int_\T[\d_x(\rho_\tau-M_0)^n]^2dx,
\]
and
\[
\int_\T(\d_xV_\tau)^2dx\leq \frac{1}{4\pi^2}\int_\T(\rho_\tau-M_0)^2dx.
\]
Thus we have
\begin{align*}
\frac{1}{2}\int_\T\rho_\tau v^2_\tau dx+\frac14\int_\T(\d_x^2\sqrt\rho_\tau)^2dx&+2\int_\T p'(\rho_\tau)(\d_x\sqrt\rho_\tau)^2dx\\
&+\frac12\int_\T(\rho_\tau-M_0)^2dx\geq \frac12[H(\rho_\tau)+E_\tau(t)].
\end{align*}
Then by letting $\tau$ small such that
\[
\frac{c_1\tau^2}{2}-\frac{2\tau^4}{\delta}-\frac{\tau^4}{4}\geq 0,
\]
namely $\tau\leq \left(\frac{2c_1\delta}{8+\delta}\right)^\frac12$, it follows \eqref{eq:dtF_2} that
\begin{equation}\label{eq:dtF_3}
\frac{d}{dt}[F_\tau(t)+\tau^2\int_\T\log\rho_\tau\d_t\rho_\tau dx]+\frac12 F_\tau(t)\leq C(M_0,E_0,\delta)F_\tau(t)\int_\T \rho_\tau v_\tau^2dx.
\end{equation}
Last, by writing
\begin{align*}
\tau^2\left|\int_\T\log\rho_\tau\d_t\rho_\tau dx\right|\leq & \tau\int_\T\rho_\tau(\log\rho_\tau)^2 dx+\tau^3\int_\T(\d_t\sqrt\rho_\tau)^2dx\\
\leq & C(M_0,E_0,\delta)\tau H(\rho_\tau)+\tau I_\tau(t),
\end{align*}
we have for $\tau< c_1$ small that
\begin{equation}\label{eq:equivF}
\frac12F_\tau(t)\leq F_\tau(t)+\tau^2\int_\T\log\rho_\tau\d_t\rho_\tau dx\leq \frac32 F_\tau(t).
\end{equation}
Thus combining \eqref{eq:dtF_3}, \eqref{eq:equivF} and using Gronwall inequality, we finish the proof of \eqref{eq:decay}.

Last, by the energy bound \eqref{eq:en_rs}, we have
\[
\int_0^t\int_\T\rho_\tau v_\tau^2dxds\leq E_0
\]
for all $0\leq t<\infty$, hence when $t\geq C(E_0,M_0,\delta)E_0$, \eqref{eq:decay} implies that $F_\tau(t)$ decay exponentially.
\end{proof}

\section{Relaxation-time limit and convergence rate}\label{sect:rxl}

In this section, we provide the rigorous justification of the relaxation-time limit as $\tau\to 0$ in the framework of GCP solutions with positive density. Moreover, an explicit rate of convergence is given with the uniform boundedness assumption of the functional $I_\tau(t)$. As the previous section, $t$ represents the rescaled time.

We assume $\{(\rho_\tau,v_\tau)\}_{\tau\in\mathcal{I}}$ to be a sequence of GCP solutions to the rescaled QHD system \eqref{eq:QHD_rs} satisfying the following conditions:
\begin{itemize}
\item[(1)] $\inf_{t,x}\rho_\tau\geq \delta>0$;
\item[(2)] there exist $M_0,E_0$ and $I_0$, such that $(\rho_\tau,v_\tau)$ satisfies the mass-energy bounds
\begin{equation}\label{eq:bd_me_rxl}
M_\tau(t)\leq  M_0,\quad E_\tau(t)+\int_0^T\int_\T \rho_\tau v_\tau^2dxds\leq E_0,
\end{equation}
and the high order energy bounds:
\begin{equation}\label{eq:bd_I_rxl}
I_\tau(t)+\int_0^t\int_\T \rho_\tau\sigma_\tau^2 dxds+\tau^2\int_0^t\int_\T \rho_\tau v_\tau^4 dxds\leq C(M_0,E_0,I_0,\delta,t),
\end{equation}
where $M_\tau(t)$, $E_\tau(t)$ and $I_\tau(t)$ are rescaled functionals given at the beginning of Section \ref{sect:rs}.
\end{itemize}

We first prove the following Proposition concerning the weak relaxation-time limit.

\begin{prop}
Let $\{(\rho_\tau,v_\tau)\}_{\tau\in\mathcal{I}}$ be a sequence of GCP solution satisfying the conditions $(1)$ and $(2)$. Then there exists a subsequence, denoted by $\{(\rho_{\tau_n},J_{\tau_n})\}$, and a limiting density $\bar\rho$ such that the relaxation-time limit holds in the sense
\[
\sqrt\rho_{\tau_n}\rightharpoonup\sqrt{\bar\rho}\quad\textrm{in }L^2_{t,loc}H^2_x,
\]
and $\bar\rho$ is a weak solution of the QDD equation \eqref{eq:qdde} defined by Definition \ref{def:qdde_ws}. Moreover, if the internal energy $f(s)$ is a convex function, then $\bar\rho$ is a dissipative solution.
\end{prop}

\begin{proof}
The conditions provided above imply $(\rho_\tau,v_\tau)$ satisfy the uniform bounds
\begin{equation}\label{eq:prop29.1}
\|\sqrt\rho_\tau\|_{L^2_tH^1_x}\leq C(M_0,E_0)T,\quad \|\sqrt\rho_\tau v_\tau\|_{L^2_{t,x}}\leq E_0^\frac12
\end{equation}
on $[0,T]\times \T$ for any $0\leq T<\infty$.
Furthermore, by \eqref{eq:en_bd1} in Proposition \ref{prop:entr} and the identity \eqref{eq:H^2bd}, we have
\begin{equation}\label{eq:7-1}
\begin{aligned}
H(\rho)(t)&+\int_0^t\int_\T(\d_x^2\sqrt\rho_\tau)^2+(\d_x\rho_\tau^\frac14)^4dxds\\
&+\int_0^t\int_\T p'(\rho_\tau)(\d_x\sqrt\rho_\tau)^2dxds\leq H(\rho_0)+C(M_0,E_0,I_0,\delta,t)\tau.
\end{aligned}
\end{equation}
For general internal energy $f\in C^2([0,\infty))$, by
\[
p(s)=s f'(s)-f(s),
\]
it follows that
\begin{align*}
\left|\int_0^t\int_\T p'(\rho_\tau)(\d_x\sqrt\rho_\tau)^2dxds\right|=&\left|\int_0^t\int_\T \rho_\tau f''(\rho_\tau)(\d_x\sqrt\rho_\tau)^2dxds\right|\\
\leq & C(M_0,E_0,\|f\|_{C^2},T).
\end{align*}
For the initial entropy $H(\rho_0)$, we have
\[
H(\rho_0)=\int_\T \rho_0\log\left(\frac{\rho_0}{M_0}\right)dx\leq C(M_0,E_0,\delta).
\]
Thus on $[0,T]\times \T$ we obtain that
\[
\|\d_x^2\sqrt\rho_\tau\|_{L^2_{t,x}}\leq C(M_0,E_0,\delta)(1+T)+C(M_0,E_0,I_0,\delta,T)\tau.
\]
Therefore we can choose a converging subsequence $\{(\rho_{\tau_n},v_{\tau_n})\}$ and a limiting function $\xi\geq 0$, such that
\[
\sqrt\rho_{\tau_n}\rightharpoonup\xi\quad\textrm{in }L^2_{t,loc}H^2_x.
\]
By the embedding theory, we also have the strong convergence
\[
\sqrt\rho_{\tau_n}\to\xi\quad\textrm{in }L^2_{t,loc}H^1_x,
\]
which implies $\rho_{\tau_n}\to \xi^2$,  thus we can write $\xi=\sqrt{\bar\rho}$. Now we take test functions $\eta\in\mathcal C^\infty_0([0, T)\times\T)$, and write the rescaled QHD system \eqref{eq:QHD_rs} in the weak form,
\[
\int_0^T\int_{\T}\rho_{\tau_n}\d_t\eta+J_{\tau_n}\d_x\eta\,dxdt+\int_{\T}\rho_0(x)\eta(0, x)\,dx=0;
\]
and
\begin{align*}
\int_0^T\int_{\T}\tau_n^2J_{\tau_n}\d_t\eta&+(\tau_n^2\Lambda_{\tau_n}^2+p(\rho_{\tau_n})+(\d_x\sqrt\rho_{\tau_n})^2)\d_x\eta-\eta\rho_\tau\d_xV_{\tau_n}
\\
&-\frac{1}{4}\rho_{\tau_n}\d_x^3\eta-J_{\tau_n}\eta\,dxdt+\int_{\T}J_0(x)\eta(0, x)\,dx=0,
\end{align*}
where $J_{\tau_n}=\sqrt\rho_{\tau_n}\Lambda_{\tau_n}=\rho_{\tau_n} v_{\tau_n}$. 
The bound \eqref{eq:prop29.1} on $\sqrt\rho_{\tau_n}v_{\tau_n}$ and the strong convergence of $\sqrt\rho_{\tau_n}$ imply the convergence 
\[
J_{\tau_n}=\rho_{\tau_n} v_{\tau_n}\rightharpoonup \bar J\quad \textrm{in }L^2_{t,loc}L^2_x,\quad V_{\tau_n}\to\bar V\quad \textrm{in }L^2_{t,loc}H^1_x.
\]
Thus we can pass to the limit ${\tau_n}\to 0$ to obtain
\[
\int_0^t\int_\T\bar\rho\d_t\eta+\bar J\d_x\eta dxds+\int_\T \rho_0\eta(0)dx=0
\]
and
\[
\int_0^t\int_\T [(\d_x\sqrt{\bar\rho})^2+p(\bar\rho)]\d_x^2\eta-\bar\rho\d_x\bar V\d_x\eta-\frac14\bar\rho\d_x^4\eta-\bar J\d_x\eta dxds=0,
\]
which gives the weak formulation of equation \eqref{eq:qdde},
\[
\int_0^t\int_\T\bar\rho\d_t\eta+[(\d_x\sqrt{\bar\rho})^2+p(\bar\rho)]\d_x^2\eta-\bar\rho\d_x\bar V\d_x\eta-\frac14\bar\rho\d_x^4\eta dxds+\int_\T \rho_0\eta(0)dx=0,
\]
as in Definition \ref{def:qdde_ws}.

Last, if we assume the internal energy $f$ to be a convex function, then we have the positivity of $p'(s)=s f''(s)$. By pass to the limit $\tau\to 0$ in \eqref{eq:7-1}, we show that $\bar\rho$ is a dispersive solution as in Definition \ref{def:qdde_ws}.
\end{proof}

Now we turn to the proof of the explicit convergence rate. As a consequence of the relaxation-time limit, the limiting density $\bar\rho$ obtained above has a lower bound $\delta>0$, and it satisfies the following bounds,
\begin{equation}\label{eq:bd_barrho_1}
\|\sqrt{\bar\rho}\|_{L^\infty_tH^1_x}\leq C(M_0,E_0),
\end{equation}
and
\begin{equation}\label{eq:bd_barrho_2}
\|\d_x^2\sqrt{\bar\rho}\|_{L^2_{t,x}}+\|\d_x\bar\rho^\frac14\|_{L^4_{t,x}}+\|\sqrt{\bar\rho}\d_x^2\log\sqrt{\bar\rho}\|_{L^2_{t,x}}\leq C(M_0,E_0,I_0,\delta,T),
\end{equation}
where the bound of $\sqrt{\bar\rho}\d_x^2\log\sqrt{\bar\rho}$ can be obtain by an analogue identity of \eqref{eq:H^2bd} for $\bar\rho$. 
Moreover, by the uniform bound of $\sqrt\rho_\tau\sigma_\tau=\d_t\sqrt\rho_\tau$, it follows that
\begin{equation}\label{eq:bd_barrho_3}
\|\d_t\sqrt{\bar\rho}\|_{L^2_{t,x}}\leq C(M_0,E_0,I_0,\delta,T).
\end{equation}
Then by \eqref{eq:bohm}, we can write the quantum drift-diffusion equation as
\[
\d_t\bar\rho+\d_x^2(\bar\rho\d_x^2\log\sqrt{\bar\rho})-\d_x^2 p(\bar\rho)-\d_x(\bar\rho\d_x\bar V)=0.
\]
by the assumption $p\in C^2((0,+\infty))$, we obtain
\begin{equation}\label{eq:bd_barrho_4}
\|\d_x^2(\bar\rho\d_x^2\log\sqrt{\bar\rho})\|_{L^2_{t,x}}\leq C(M_0,E_0,I_0,\delta,T).
\end{equation}

As the main technical tool of this part, we use the method of relative entropy. Recall that the physical entropy $H(\rho)(t)$ is given by
\[
H(\rho)(t)=\int_\T g(\rho(t))dx,
\]
where $g(s)=s \log(s/M_0)$, then for $\rho_\tau$ and the limiting density $\bar\rho$, we define the relative entropy as
\begin{equation}\label{eq:rl_entr}
\begin{aligned}
H(\rho_\tau|\bar\rho)(t)=&H(\rho_\tau)(t)-H(\bar\rho)(t)-\int_\T g'(\bar\rho)(\rho_\tau-\bar\rho)dx\\
=&\int_\T g(\rho_\tau)-g(\bar\rho)-g'(\bar\rho)(\rho_\tau-\bar\rho)dx
\end{aligned}
\end{equation} 
with
\[
g'(s)=\log\left(\frac{s}{M_0}\right)+1.
\]
By Taylor's formula, we have
\[
g(\rho_\tau)-g(\bar\rho)-g'(\bar\rho)(\rho_\tau-\bar\rho)=\frac12 g''(\rho_*)(\rho_\tau-\bar\rho)^2
\]
for some $\rho_*\in [\min\{\rho_\tau,\bar\rho\},\max\{\rho_\tau,\bar\rho\}]$, and $g''(s)=s^{-1}$. Moreover, densities $\rho_\tau$ and $\bar\rho$ have a lower bound $\delta>0$, and have an upper bound $C(M_0,E_0)$ by the mass-energy bound \eqref{eq:bd_me_rxl} and \eqref{eq:bd_barrho_1}. Thus the relative entropy $H(\rho_\tau|\bar\rho)$ is equivalent to the $L^2_x$ norm of $\rho_\tau-\bar\rho$ in the sense
\begin{equation}\label{eq:equiv_entr}
2 \delta H(\rho_\tau|\bar\rho)(t)\leq\|\rho_\tau-\bar\rho\|_{L^2_x}^2(t)\leq 2 C(M_0,E_0)H(\rho_\tau|\bar\rho)(t).
\end{equation}

To estimate the relative entropy $H(\rho_\tau|\bar\rho)$, we need the following result which gives the time derivative of $H(\rho_\tau|\bar\rho)$. Here again a rigorous computation need the argument of mollification of $\rho_\tau$ as in the proof of Proposition \ref{prop:entr}. Since the method is similar, we omit this step for simplicity of the paper.

\begin{prop}\label{prop:6.2}
The time derivative of $H(\rho_\tau|\bar\rho)$ is given by
\begin{equation}\label{eq:dtrlentr}
\begin{aligned}
\frac{d}{dt}H(\rho_\tau|\bar\rho)=&-\int_\T\rho_\tau(\d_x^2\log\sqrt\rho_\tau-\d_x^2\log\sqrt{\bar\rho})^2dx\\
&+\int_\T \rho_\tau\d_x^2\log\sqrt{\bar\rho}(\d_x\log\sqrt\rho_\tau-\d_x\log\sqrt{\bar\rho})^2dx\\
&-\int_\T \rho_\tau[p'(\rho_\tau)\d_x\log\rho_\tau-p'(\bar\rho)\d_x\log\bar\rho](\d_x\log\rho_\tau-\d_x\log\bar\rho) dx\\
&-\int_\T(\rho_\tau-\bar\rho)^2dx+\int_\T(\rho_\tau-\bar\rho)\d_x\log\bar\rho\d_x(V_\tau-\bar V)dx\\
&+\tau^2[R(\rho_\tau)-R(\rho_\tau,\bar\rho)],
\end{aligned}
\end{equation}
where the remaining terms are given by
\[
R(\rho_\tau)=-\frac{d}{dt}\int_\T\log\rho_\tau\d_t\rho_\tau dx+4\int_\T\rho_\tau\sigma_\tau^2 dx+\int_\T\rho_\tau v_\tau^2\d_x^2\log\rho_\tau dx,
\]
and
\[
R(\rho_\tau,\bar\rho)=-\frac{d}{dt}\int_\T\log\bar\rho\d_t\rho_\tau dx+\int_\T\d_t\log\bar\rho\d
_t\rho_\tau dx+\int_\T\rho_\tau v_\tau^2\d_x^2\log\bar\rho dx.
\]
\end{prop}

\begin{proof}
The time derivative of $H(\rho_\tau|\bar\rho)$ is computed by
\[
\frac{d}{dt}H(\rho_\tau|\bar\rho)=\frac{d}{dt}H(\rho_\tau)-\int_\T \d_t\rho_\tau\log\bar\rho dx-\int_\T \rho_\tau\d_t\log\bar\rho dx,
\]
where the terms containing $\int_\T\d_t\rho_\tau$ and $\int_\T\d_t\bar\rho$ are all erased due to the equations. The time derivative $\frac{d}{dt}H(\rho_\tau)$ has been given in Proposition \ref{prop:entr} as
\begin{align*}
\frac{d}{dt}H(\rho_\tau)=&-\int_\T\rho_\tau(\d_x^2\log\sqrt\rho_\tau)^2 dx-\int_\T p'(\rho_\tau)\rho_\tau(\d_x\log\rho_\tau)^2dx\\
&-\int_\T\rho_\tau(\rho_\tau-\mathcal{C}(x))dx+\tau^2 R(\rho_\tau).
\end{align*}
For the remaining terms, by using \eqref{eq:QHD_rs} and \eqref{eq:qdde}, we have
\begin{align*}
-\int_\T \d_t\rho_\tau\log\bar\rho dx=&\frac12\int_\T\log\bar\rho\d_x^2(\rho_\tau\d_x^2\log\sqrt\rho_\tau)dx-\int_\T\log\bar\rho\d_x^2p(\rho_\tau)dx\\
&-\int_\T\log\bar\rho\d_x(\rho_\tau\d_xV_\tau) dx+\tau^2\int_\T\log\bar\rho\d_t^2\rho_\tau dx\\
&-\tau^2\int_\T\log\bar\rho\d_x^2(\rho_\tau v_\tau^2)dx\\
=&\int_\T\rho_\tau\d_x^2\log\sqrt{\bar\rho}\d_x^2\log\sqrt\rho_\tau+\int_\T p'(\rho_\tau)\rho_\tau\d_x\log\rho_\tau\d_x\log\bar\rho dx\\
&+\int_\T\rho_\tau\d_x\log\bar\rho\d_xV_\tau dx-\tau^2 R(\rho_\tau,\bar\rho),
\end{align*}
where 
\begin{align*}
R(\rho_\tau,\bar\rho)=&\int_\T\rho_\tau v_\tau^2\d_x^2\log\bar\rho dx-\int_\T\log\bar\rho\d_t^2\rho_\tau dx\\
=&\int_\T\rho_\tau v_\tau^2\d_x^2\log\bar\rho dx-\frac{d}{dt}\int_\T\log\bar\rho\d_t\rho_\tau dx+\int_\T\d_t\log\bar\rho\d_t\rho_\tau dx,
\end{align*}
and
\begin{align*}
-\int_\T \rho_\tau\d_t\log\bar\rho dx=&\frac12\int_\T \frac{\rho_\tau}{\bar\rho}\d_x^2(\bar\rho\d_x^2\log\sqrt{\bar\rho})dx-\int_\T \frac{\rho_\tau}{\bar\rho}\d_x^2p(\bar\rho)dx\\
&-\int_\T \frac{\rho_\tau}{\bar\rho}\d_x(\bar\rho\d_x\bar V)dx\\
=&\frac12\int_\T \bar\rho\d_x^2\log\sqrt{\bar\rho}\d_x^2\left(\frac{\rho_\tau}{\bar\rho}\right)dx+\int_\T \d_xp(\bar\rho)\d_x\left(\frac{\rho_\tau}{\bar\rho}\right)dx\\
&+\int_\T \rho_\tau(\bar\rho-\mathcal{C}(x))dx-\int_\T \rho_\tau\d_x\log\bar\rho\d_x\bar Vdx.
\end{align*}
By substituting 
\[
\d_x\left(\frac{\rho_\tau}{\bar\rho}\right)=\frac{\rho_\tau}{\bar\rho}(\d_x\log\rho_\tau-\d_x\log\bar\rho)
\]
and
\[
\d_x^2\left(\frac{\rho_\tau}{\bar\rho}\right)=\frac{\rho_\tau}{\bar\rho}(\d_x^2\log\rho_\tau-\d_x^2\log\bar\rho)+\frac{\rho_\tau}{\bar\rho}(\d_x\log\rho_\tau-\d_x\log\bar\rho)^2
\]
into the previous identity, we obtain
\begin{align*}
-\int_\T \rho_\tau\d_t\log\bar\rho dx=&\int_\T \rho_\tau\d_x^2\log\sqrt{\bar\rho}(\d_x^2\log\sqrt\rho_\tau-\d_x^2\log\sqrt{\bar\rho})dx\\
&+\int_\T \rho_\tau\d_x^2\log\sqrt{\bar\rho}(\d_x\log\sqrt\rho_\tau-\d_x\log\sqrt{\bar\rho})^2dx\\
&+\int_\T p'(\bar\rho)\rho_\tau\d_x\log\bar\rho(\d_x\log\rho_\tau-\d_x\log\bar\rho)dx.
\end{align*}
To obtain \eqref{eq:dtrlentr}, we summarize the computations above, and notice that the integrals containing $V_\tau$, $\bar V$ and $\mathcal{C}(x)$ are
\begin{align*}
&-\int_\T\rho_\tau(\rho_\tau-\bar\rho)dx+\int_\T\rho_\tau\d_x\log\bar\rho\d_x(V_\tau-\bar V)dx\\
=&-\int_\T\rho_\tau(\rho_\tau-\bar\rho)dx+\int_\T(\rho_\tau-\bar\rho)\d_x\log\bar\rho\d_x(V_\tau-\bar V)dx+\int_\T\d_x\bar\rho\d_x(V_\tau-\bar V)dx\\
=&-\int_\T\rho_\tau(\rho_\tau-\bar\rho)dx+\int_\T(\rho_\tau-\bar\rho)\d_x\log\bar\rho\d_x(V_\tau-\bar V)dx-\int_\T\bar\rho\d_x^2(V_\tau-\bar V)dx\\
=&-\int_\T(\rho_\tau-\bar\rho)^2dx+\int_\T(\rho_\tau-\bar\rho)\d_x\log\bar\rho\d_x(V_\tau-\bar V)dx.
\end{align*}
\end{proof}

\begin{proof}[Proof of Theorem \ref{thm:rlxlimit}]
The proof is finished by controlling the difference terms in the right hand side of \eqref{eq:dtrlentr} by $H(\rho_\tau|\bar\rho)$ and the dissipation. 

We first write
\begin{align*}
\int_\T \rho_\tau\d_x^2\log\sqrt{\bar\rho}&(\d_x\log\sqrt\rho_\tau-\d_x\log\sqrt{\bar\rho})^2dx\\
\leq & \delta^{-1}\|\rho_\tau\|_{L^\infty_x}\|\bar\rho\d_x^2\log\sqrt{\bar\rho}\|_{L^\infty_x}\|\d_x\log\sqrt\rho_\tau-\d_x\log\sqrt{\bar\rho}\|^2_{L^2_x}.
\end{align*}
It follows interpolation inequality that
\begin{equation}\label{eq:bd_log1}
\|\d_x\log\sqrt\rho_\tau-\d_x\log\sqrt{\bar\rho}\|^2_{L^2_x}\leq \|\log\sqrt\rho_\tau-\log\sqrt{\bar\rho}\|_{L^2_x}\|\d_x^2\log\sqrt\rho_\tau-\d_x^2\log\sqrt{\bar\rho}\|_{L^2_x}.
\end{equation}
Moreover, by using
\begin{equation}\label{eq:bd_log2}
|\log\sqrt\rho_\tau-\log\sqrt{\bar\rho}|\leq \frac{|\rho_\tau-\bar\rho|}{2\min\{\rho_\tau,\bar\rho\}}\leq \frac{|\rho_\tau-\bar\rho|}{2\delta},
\end{equation}
we obtain that
\begin{align*}
\int_\T \rho_\tau\d_x^2\log\sqrt{\bar\rho}&(\d_x\log\sqrt\rho_\tau-\d_x\log\sqrt{\bar\rho})^2dx\\
\leq & C(M_0,E_0,\delta)\|\bar\rho\d_x^2\log\sqrt{\bar\rho}\|_{L^\infty_x}\\
&\cdot \|\rho_\tau-\bar\rho\|_{L^2_x}\|\sqrt\rho_\tau(\d_x^2\log\sqrt\rho_\tau-\d_x^2\log\sqrt{\bar\rho})\|_{L^2_x}\\
\leq & C(M_0,E_0,\delta)\|\bar\rho\d_x^2\log\sqrt{\bar\rho}\|_{L^\infty_x}^2 \|\rho_\tau-\bar\rho\|_{L^2_x}^2\\
&+\frac14\|\sqrt\rho_\tau(\d_x^2\log\sqrt\rho_\tau-\d_x^2\log\sqrt{\bar\rho})\|_{L^2_x}^2.
\end{align*}
By Gagliardo-Nirenberg  inequality, we further write
\begin{align*}
\|\bar\rho\d_x^2\log\sqrt{\bar\rho}\|_{L^\infty_x}^2\leq & C\|\bar\rho\d_x^2\log\sqrt{\bar\rho}\|_{L^2_x}^\frac32\|\d_x^2(\bar\rho\d_x^2\log\sqrt{\bar\rho})\|_{L^2_x}^\frac12+C\|\bar\rho\d_x^2\log\sqrt{\bar\rho}\|_{L^2_x}^2\\
\leq & C(M_0,E_0)(\|\sqrt{\bar\rho}\d_x^2\log\sqrt{\bar\rho}\|_{L^2_x}^2+\|\d_x^2(\bar\rho\d_x^2\log\sqrt{\bar\rho})\|_{L^2_x}^2).
\end{align*}
Thus by the equivalence \eqref{eq:equiv_entr} of $H(\rho_\tau|\bar\rho)$ and $\|\rho_\tau-\bar\rho\|_{L^2_x}^2$, we obtain
\begin{align*}
\int_\T \rho_\tau\d_x^2\log\sqrt{\bar\rho}&(\d_x\log\sqrt\rho_\tau-\d_x\log\sqrt{\bar\rho})^2dx\\
\leq & a_1(t)H(\rho_\tau|\bar\rho)+\frac14\int_\T\rho_\tau(\d_x^2\log\sqrt\rho_\tau-\d_x^2\log\sqrt{\bar\rho})^2 dx,
\end{align*}
where
\[
a_1(t)=C(M_0,E_0,\delta)(\|\sqrt{\bar\rho}\d_x^2\log\sqrt{\bar\rho}\|_{L^2_x}^2(t)+\|\d_x^2(\bar\rho\d_x^2\log\sqrt{\bar\rho})\|_{L^2_x}^2(t)).
\]

To estimate the pressure term in the right hand side of \eqref{eq:dtrlentr}, we write
\begin{align*}
p'(\rho_\tau)\d_x\log\rho_\tau-p'(\bar\rho)\d_x\log\bar\rho=&[p'(\rho_\tau)-p'(\bar\rho)]\d_x\log\rho_\tau\\
&+p'(\bar\rho)(\d_x\log\rho_\tau-\d_x\log\bar\rho),
\end{align*}
which gives
\begin{align*}
-\int_\T \rho_\tau&[p'(\rho_\tau)\d_x\log\rho_\tau-p'(\bar\rho)\d_x\log\bar\rho](\d_x\log\rho_\tau-\d_x\log\bar\rho) dx\\
\leq & \int_\T |\d_x\rho_\tau[p'(\rho_\tau)-p'(\bar\rho)](\d_x\log\rho_\tau-\d_x\log\bar\rho) |dx\\
&+\int_\T\rho_\tau|p'(\bar\rho)|(\d_x\log\rho_\tau-\d_x\log\bar\rho)^2dx=\textrm{I}+\textrm{II}.
\end{align*}
By $p\in C^2[0,+\infty)$ and writing 
\[
p'(\rho_\tau)-p'(\bar\rho)=p''(\rho_*)(\rho_\tau-\bar\rho),
\]
we estimate the integral I as
\begin{align*}
\textrm{I}\leq & 2\|\sqrt\rho_\tau\|_{L^\infty_x}\|\d_x\sqrt\rho_\tau\|_{L^\infty_x}\|p''(\rho_*)\|_{L^\infty_x}\|\rho_\tau-\bar\rho\|_{L^2_x}\|\d_x\log\rho_\tau-\d_x\log\bar\rho\|_{L^2_x}\\
\leq & C(M_0,E_0,\delta)\|\d_x\sqrt\rho_\tau\|_{L^\infty_x}\|\rho_\tau-\bar\rho\|_{L^2_x}^\frac32\|\sqrt\rho_\tau(\d_x\log\sqrt\rho_\tau-\d_x\log\sqrt{\bar\rho})\|_{L^2_x}^\frac12\\
\leq & C(M_0,E_0,\delta)\|\d_x\sqrt\rho_\tau\|_{L^\infty_x}^\frac43\|\rho_\tau-\bar\rho\|_{L^2_x}^2+\frac14\|\sqrt\rho_\tau(\d_x\log\sqrt\rho_\tau-\d_x\log\sqrt{\bar\rho})\|_{L^2_x}^2,
\end{align*}
where in the second inequality, we again use \eqref{eq:bd_log1} and \eqref{eq:bd_log2}, and by using
\[
\|\d_x\sqrt\rho_\tau\|_{L^\infty_x}\leq C\|\d_x\sqrt\rho_\tau\|_{L^2_x}^\frac12\|\d_x^2\sqrt\rho_\tau\|_{L^2_x}^\frac12\leq C(E_0)\|\d_x^2\sqrt\rho_\tau\|_{L^2_x}^\frac12,
\]
we obtain
\[
I\leq a_2(t)H(\rho_\tau|\bar\rho)+\frac14\int_\T\rho_\tau(\d_x^2\log\sqrt\rho_\tau-\d_x^2\log\sqrt{\bar\rho})^2 dx
\]
with 
\[
a_2(t)=C(M_0,E_0,\delta)\|\d_x^2\sqrt\rho_\tau\|_{L^2_x}^\frac23(t).
\]
Last, for the term \textrm{II}, by $p\in C^2[0,+\infty)$, \eqref{eq:bd_log1} and \eqref{eq:bd_log2} it follows a similar argument as before that
\begin{align*}
\textrm{II}\leq & C(M_0,E_0,\delta)H(\rho_\tau|\bar\rho)+\frac14\int_\T\rho_\tau(\d_x^2\log\sqrt\rho_\tau-\d_x^2\log\sqrt{\bar\rho})^2 dx.
\end{align*}
For the integral of electric potentials, we control it by
\begin{align*}
\int_\T(\rho_\tau-\bar\rho)\d_x\log\bar\rho\d_x(V_\tau-\bar V)dx\leq & \|\d_x\log\bar\rho\|_{L^2_x}\|\rho_\tau-\bar\rho\|_{L^2_x}\|\d_x(V_\tau-\bar V)\|_{L^\infty_x}\\
\leq& C(M_0,E_0,\delta)\|\rho_\tau-\bar\rho\|_{L^2_x}^2\sim C(M_0,E_0,\delta)H(\rho_\tau|\bar\rho)(t),
\end{align*}
where we use the Poisson equation for $V_\tau$, $\bar V$ and Poincar\'e inequality
\[
\|\d_x(V_\tau-\bar V)\|_{L^\infty_x}\leq C\|\d_x^2(V_\tau-\bar V)\|_{L^2_x}=C\|\rho_\tau-\bar\rho\|_{L^2_x}.
\]

Summarizing the inequalities above and integrating in time , since $H(\rho_\tau|\bar\rho)(0)=0$ we have
\begin{align*}
H(\rho_\tau|\bar\rho)(t)+\frac14\int_0\int_\T\rho_\tau &(\d_x^2\log\sqrt\rho_\tau-\d_x^2\log\sqrt{\bar\rho})^2dxds\\
\leq & \int_0^t a(s)H(\rho_\tau|\bar\rho)ds+\tau^2\int_0^t[R(\rho_\tau)-R(\rho_\tau,\bar\rho)]ds,
\end{align*}
where
\[
a(t)=C(M_0,E_0,\delta)[a_1(t)+a_2(t)+1],
\]
and by the bounds \eqref{eq:bd_barrho_2} and \eqref{eq:bd_barrho_4}, it follows that
\begin{align*}
\int_0^t a(s)ds=& C(M_0,E_0,\delta)\left(\|\sqrt{\bar\rho}\d_x^2\log\sqrt{\bar\rho}\|_{L^2_{t,x}}^2+\|\d_x^2(\bar\rho\d_x^2\log\sqrt{\bar\rho})\|_{L^2_{t,x}}^2\right)\\
&+C(M_0,E_0,\delta)\left(\int_0^t \|\d_x^2\sqrt\rho_\tau\|_{L^2_x}^\frac23(s)ds+t\right)\\
\leq & C(M_0,E_0,I_0,\delta,T).
\end{align*}

Now we estimate the last term $\tau^2\int_0^t[R(\rho_\tau)-R(\rho_\tau,\bar\rho)]ds$. Recall from Proposition \ref{prop:6.2} that 
\begin{equation}\label{eq:6.13}
\begin{aligned}
\tau^2\int_0^t[R(\rho_\tau)-R(\rho_\tau,\bar\rho)]ds=&-\tau^2\left.\int_\T(\log\rho_\tau-\log\bar\rho)\d_t\rho_\tau dx\right|_{s=0}^{s=t}\\
&+\tau^2\int_0^t\int_\T(\d_x^2\log\rho_\tau-\d_x^2\log\bar\rho)\rho_\tau v_\tau^2dxds\\
&+\tau^2\int_0^t\int_\T(4\rho_\tau\sigma^2_\tau+\d_t\log\bar\rho\d_t\rho_\tau) dxds.
\end{aligned}
\end{equation}
By Cauchy-Schwartz inequality, the first term is controlled by
\begin{align*}
\tau^2\left|\int_\T(\log\rho_\tau-\log\bar\rho)\d_t\rho_\tau dx\right|\leq & C\tau^4\|\sqrt\rho_\tau\sigma_\tau\|_{L^2_x}^2+C^{-1}\|\sqrt\rho_\tau(\log\rho_\tau-\log\bar\rho)\|_{L^2_x}^2.
\end{align*}
It follows from \eqref{eq:bd_I_rxl} that
\[
\tau^4\|\sqrt\rho_\tau\sigma_\tau\|_{L^2_x}^2\leq C(M_0,E_0,I_0,\delta,t)\tau^2,
\]
and by the upper and lower bound of $\rho_\tau$ and $\bar\rho$, we have
\[
\|\sqrt\rho_\tau(\log\rho_\tau-\log\bar\rho)\|_{L^2_x}^2\sim\|\rho_\tau-\bar\rho\|_{L^2_x}^2\sim H(\rho_\tau|\bar\rho).
\]
Therefore we can choose $C$ large such that
\[
\tau^2\left|\int_\T(\log\rho_\tau-\log\bar\rho)\d_t\rho_\tau dx\right|\leq C(M_0,E_0,I_0,\delta,t)\tau^2+\frac12 H(\rho_\tau|\bar\rho).
\]
For the second term in the right hand side of \eqref{eq:6.13}, we have
\begin{align*}
\tau^2\int_0^t\int_\T&(\d_x^2\log\rho_\tau-\d_x^2\log\bar\rho)\rho_\tau v_\tau^2dxds\\
\leq&\frac18\int_0^t\int_\T\rho_\tau (\d_x^2\log\sqrt\rho_\tau-\d_x^2\log\sqrt{\bar\rho})^2dxds+4\tau^4\int_0^t\int_\T\rho_\tau v_\tau^4dxds,
\end{align*}
and by \eqref{eq:bd_I_rxl}, the last term is bounded by 
\[
\tau^4\int_0^t\int_\T\rho_\tau v_\tau^4dxds\leq C(M_0,E_0,I_0,\delta,t)\tau^2.
\]
The last term in \eqref{eq:6.13} are control by \eqref{eq:bd_I_rxl} and \eqref{eq:bd_barrho_3} as
\begin{align*}
\tau^2\int_0^t\int_\T&(4\rho_\tau\sigma^2_\tau+\d_t\log\bar\rho\d_t\rho_\tau) dxds\\
\leq & C\tau^2\int_0^t(\|\sqrt\rho_\tau\sigma_\tau\|_{L^2_x}^2+\|\sqrt\rho_\tau\sigma_\tau\|_{L^2_x}\|\d_t\sqrt{\bar\rho}\|_{L^2_x})ds\leq C(M_0,E_0,I_0,\delta,t)\tau^2.
\end{align*}

Summarizing the arguments above, we obtain
\begin{align*}
\frac12H(\rho_\tau|\bar\rho)(t)+\frac18\int_0\int_\T\rho_\tau &(\d_x^2\log\sqrt\rho_\tau-\d_x^2\log\sqrt{\bar\rho})^2dxds\\
\leq & \int_0^t a(s)H(\rho_\tau|\bar\rho)ds+C(M_0,E_0,I_0,\delta,t)\tau^2,
\end{align*}
then by the integral form of Gronwall inequality and $H(\rho_\tau|\bar\rho)(0)=0$, we obtain
\[
H(\rho_\tau|\bar\rho)(t)+\int_0^t\int_\T\rho_\tau (\d_x^2\log\sqrt\rho_\tau-\d_x^2\log\sqrt{\bar\rho})^2dxds\leq C(M_0,E_0,I_0,\delta,t)\tau^2
\]
for all $0\leq t<\infty$. By the equivalence \eqref{eq:equiv_entr} of $H(\rho_\tau|\bar\rho)$ and $\|\rho_\tau-\bar\rho\|_{L^2_x}$, we conclude the proof.
\end{proof}

\appendix

\section{Hydrodynamical computation of identities \eqref{eq:en_disp} and \eqref{eq:rl_I}}

For completeness of this paper, in this appendix we derive the energy balance law \eqref{eq:en_disp} and the time derivative \eqref{eq:rl_I} of the higher order functional $I(t)$ in a purely hydrodynamic way from the dynamics of the QHD system \eqref{eq:QHD_1d}.
In particular, for nowhere vanishing, GCP solutions we can equivalently write system \eqref{eq:QHD_1d} as
\begin{equation}\label{eq:QHD_apdx}
\left\{\begin{aligned}
&\d_t\rho +\d_x(\rho  v )=0\\
&\d_tv + v \d_xv +\d_xf'(\rho )+\d_x V 
=\frac12\d_x\left(\frac{\d_x^2\sqrt{\rho }}{\sqrt{\rho }}\right)-\frac{1}{\tau}v. 
\end{aligned}\right.
\end{equation}
In this way, identity \eqref{eq:en_disp} corresponds to the classical energy estimate for \eqref{eq:QHD_apdx}, while to prove \eqref{eq:rl_I} we are going to compute the evolution equations for $\mu$ and $\sigma$ first, see \eqref{eq:mu_evol_apdx} and \eqref{eq:sigma_evol_apdx} below.
\newline
In the next proposition, we are actually going to show that for GCP solutions we can rigorously derive the local balance law for the energy density.
%In what follows, we are going to show that identity \eqref{eq:en_disp} can be rigorously justified for GCP solutions to \eqref{eq:QHD_1d}, while for \eqref{eq:rl_I} we provide a formal calculation that holds for smooth, non-vanishing solutions.

\begin{prop}\label{lem:dte_apdx}
Let $(\rho, v)$ be a GCP solution to \eqref{eq:QHD_1d} with finite energy, such that $\rho>0$. Then the energy density 
\[
e(t,x)=\frac12(\d_x\sqrt\rho)^2+\frac12\rho v^2+f(\rho)+\frac12(\d_xV)^2
\]
satisfies the following distributional equation
\begin{equation}\label{eq:en_cons_apdx}
\d_te+\d_x(\rho v\mu-\d_t\sqrt{\rho}\d_x\sqrt{\rho}-V\d_x\d_tV)+\frac{1}{\tau}\rho v^2=0.
\end{equation}
As a consequence, the total energy $E(t)=\int_\T e(t,x)dx$ satisfies the energy balance law
\begin{equation}\label{eq:en_disp_apdx}
E(t)+\frac{1}{\tau}\int_0^t\int_\T\rho v^2dxds=E_0
\end{equation}
\end{prop}
\begin{proof}
We notice from \eqref{eq:QHD_apdx} that the equation for the velocity field can be equivalently written as follows
\begin{equation*}
\d_tv +\d_x\mu +\frac{1}{\tau}v =0,
\end{equation*}
where $\mu $ is the chemical potential defined in \eqref{eq:chem}, and its derivative is interpreted in the $H^{-1}_x$ sense.
By using the expression of $e(t,x)$, we can differentiate it with respect to time and find
\begin{equation*}
\begin{aligned}
\d_te =&\d_x\sqrt{\rho }\d_{t}\d_x\sqrt{\rho }+\left(\frac{1}{2}v^2 +f'(\rho )\right)\d_t\rho +\rho  v \d_tv +\d_xV\d_x\d_tV\\
=&\d_x\left(\d_x\sqrt{\rho }\d_t\sqrt{\rho }\right)
+\left(-\frac12\frac{\d_x^2\sqrt{\rho }}{\sqrt{\rho }}+\frac{1}{2}v ^2+f'(\rho )+V \right)\d_t\rho \\
&+\rho  v \d_tv +V \d_x^2\d_tV +\d_xV\d_x\d_tV.
\end{aligned}
\end{equation*}
By using the QHD equations above and definition \eqref{eq:chem} we then have
\begin{equation*}\begin{aligned}
\d_te =&\d_x\left(\d_x\sqrt{\rho }\d_t\sqrt{\rho }\right)-\mu \d_x(\rho  v )\\
&+\d_x(V \d_x\d_tV )-\rho  v \d_x\mu -\frac{1}{\tau}\rho  v ^2\\
=&\d_x\left(\d_x\sqrt{\rho }\d_t\sqrt{\rho }-\rho  v \mu +V \d_x\d_tV \right)-\frac{1}{\tau}\rho  v^2 .
\end{aligned}
\end{equation*}
Last, the energy balance law \eqref{eq:en_disp_apdx} is obtained by integrating \eqref{eq:en_cons_apdx} on $[0,t]\times\mathbb T$.
\end{proof}

We now compute the time derivative of the functional $I(t)$ defined by \eqref{eq:higher}. However, unlike the results in Section \ref{sect:exist}, which are rigorously obtained by using the mollification of wave functions, we notice that the regularity of GCP solutions is insufficient to pass the limit in the process of mollifying the hydrodynamic functions such as $\rho $, $v $ and $\mu $. It will be another interesting question to figure out whether there exists an essential difference between the formulation of wave functions and hydrodynamic functions. Here we only present the formal computation for smooth solutions.

\begin{prop}\label{prop:dtI_apdx}
Let $(\rho , v )$ be a smooth solution to \eqref{eq:QHD_1d} such that $\rho >0$. Then the time derivative of $I (t)$ is given by
\begin{align*}
\frac{d}{dt}I (t)+\frac{1}{\tau}\int_\T \rho \sigma ^2 dx=&\int_\T\mu \d_tp(\rho )dx+\int_\T\rho \mu \d_tV  dx-\frac{1}{\tau}\int_\T\rho  v ^2\mu  dx,
\end{align*}
\end{prop}

\begin{proof}
By using the formula 
\begin{equation*}
\rho \mu =-\frac14\d_x^2\rho +e +p(\rho )-\frac12(\d_xV )^2+\rho  V 
\end{equation*}
and using \eqref{eq:en_cons_apdx}, we have
\begin{equation*}
\begin{aligned}
\d_t(\rho \mu )=&\d_t\left(e -\frac14\d_x^2\rho +p(\rho )-\frac12(\d_xV )^2+\rho  V \right)\\
=&\d_x(\d_x\sqrt{\rho }\d_t\sqrt{\rho }-\rho  v \mu +V \d_x\d_tV )\\
&-\frac14\d_{x}^2\d_t\rho -\frac{1}{\tau}\rho  v^2 +\d_tp(\rho )-\d_xV \d_x\d_tV +V \d_t\rho +\rho \d_tV .
\end{aligned}
\end{equation*}
Again by using the continuity equation of $\rho$ and the Poisson equation of $V $, we can write
\begin{equation}\label{eq:mu_evol_apdx}
\begin{aligned}
\rho \d_t\mu +\rho  v \d_x\mu =\d_x(\d_x\sqrt{\rho }\d_t\sqrt{\rho })-\frac14\d_x^2\d_t\rho \\
-\frac{1}{\tau}\rho  v^2 +\d_tp(\rho )+\rho \d_tV .
\end{aligned}
\end{equation}
Now, to write the equation for $\sigma $ we may proceed in the following way. By writing the continuity equation as below
\begin{equation*}
\d_t\rho +v \d_x\rho +\rho \d_xv =0,
\end{equation*}
we find the equation for $\log\sqrt{\rho }$, namely
\begin{equation*}
\d_t\log\sqrt{\rho }+v\d_x\log\sqrt{\rho }+\frac12\d_xv =0.
\end{equation*}
By differentiating the last equation with respect to time and by using $\d_tv =-\d_x\mu -\frac{1}{\tau}v $ we then obtain
\begin{equation*}
\d_t\sigma +v \d_x\sigma -\left(\d_x\mu +\frac{1}{\tau}v\right)\d_x\log\sqrt{\rho }-\frac12\d_x\left(\d_x\mu +\frac{1}{\tau}v\right)=0.
\end{equation*}
By multiplying this by $\rho $ and using
\[
\rho  v \d_x\log\sqrt\rho +\frac12\rho \d_xv =\frac12\d_x(\rho  v )=-\rho \sigma ,
\]
we get
\begin{equation}\label{eq:sigma_evol_apdx}
\rho \d_t\sigma +\rho  v \d_x\sigma +\frac{1}{\tau}\rho \sigma =\frac12\d_x(\rho \d_x\mu ).
\end{equation}
Now we can use the equations \eqref{eq:mu_evol_apdx} and \eqref{eq:sigma_evol_apdx} to compute the time derivative of the functional $I (t)$. After integrating by parts we find out that
\begin{align*}
\frac{d}{dt}I (t)=&\frac{1}{2}\int_\T (\mu ^2+\sigma ^2)\d_t\rho  dx-\frac{1}{2}\int_\T \rho  v \d_x(\mu ^2+\sigma ^2)dx\\
&-\frac{1}{2}\int_\T \mu \d_x(\rho\d_x\sigma )dx+\frac{1}{2}\int_\T\sigma \d_x(\rho \d_x\mu )dx\\
&+\int_\T \mu \d_tp(\rho )dx+\int_\T \rho \mu \d_tV  dx-\frac{1}{\tau}\int_\T \rho \sigma ^2dx-\frac{1}{\tau}\int_\T\rho  v ^2\mu  dx\\
=&\int_\T \mu \d_tp(\rho )dx+\int_\T \rho \mu \d_tV  dx-\frac{1}{\tau}\int_\T \rho \sigma ^2dx-\frac{1}{\tau}\int_\T\rho  v ^2\mu dx.
\end{align*}
\end{proof}

\section*{Acknowledgements}
This work was partially supported by the Strategic Priority Research Program of the Chinese Academy of Sciences under the grant XDB0510201 (H.Z.) and by the Italian Ministry of University and Research (MUR) 
%and the European Union - NextGenerationEU, 
within the PRIN project with grant number 2022YXWSLR (P.A.).
The first two authors acknowledge partial support from INdAM-GNAMPA. P.A. is also partially supported by PRIN project 20204NT8W4.

\bigskip
		
		\noindent\textbf{Data availability statement.} This work has no associated data.
		
		\medskip
		
		\noindent\textbf{Conflicts of interest.} The authors have no conflicts of interest to declare that are relevant to the content of this article.

\end{document}